\documentclass[leqno]{amsart}
\usepackage{amsmath}
\usepackage{amssymb}
\usepackage{amsthm}
\usepackage{enumerate} 
\usepackage[mathscr]{eucal}
\usepackage{xfrac}
\theoremstyle{plain}
\usepackage{tikz}
\newtheorem{theorem}{Theorem}[section]
\newtheorem{lemma}[theorem]{Lemma}
\newtheorem{prop}[theorem]{Proposition}
\theoremstyle{definition}
\newtheorem{definition}[theorem]{Definition}
\newtheorem{remark}[theorem]{Remark}
\newtheorem{conjecture}[theorem]{Conjecture}
\newtheorem{example}[theorem]{Example}

\newtheorem{cor}[theorem]{Corollary}
\theoremstyle{remark}  
\usepackage[latin1]{inputenc}
\usepackage{tikz}
\usetikzlibrary{shapes,arrows}

\begin{document}
	
	\title[On $k$-smoothness of operators between Banach spaces]{On $k$-smoothness of operators between Banach spaces}
\author[D. Sain, S. Sohel and K.Paul]{Debmalya Sain, Shamim Sohel and Kallol Paul }

\address[Sain]{Department of Mathematics\\ Indian Institute of Information Technology, Raichur\\ Karnataka 584135 \\INDIA}
\email{saindebmalya@gmail.com}

\address[Sohel]{Department of Mathematics, Jadavpur University, Kolkata 700032, India.}
\email{shamimsohel11@gmail.com}

\address[Paul]{Department of Mathematics, Jadavpur University, Kolkata 700032, India.}
\email{kalloldada@gmail.com}

	\begin{abstract}
We explore the $k$-smoothness of bounded linear operators between Banach spaces, using the newly introduced notion of index of smoothness. The characterization of the $k$-smoothness of operators between Hilbert spaces follows as a direct consequence of our study. We also investigate the $k$-smoothness of operators between polyhedral Banach spaces. In particular, we show that the $k$-smoothness of rank $1$ operators between polyhedral spaces depends heavily on the dimension of the corresponding spaces rather than the geometry of the spaces. The results obtained in this article generalize and improve upon the existing results in the $k$-smoothness of operators between Banach spaces.
	\end{abstract}

	
	\thanks{ Debmalya Sain feels elated to acknowledge the inspiring presence of his childhood friend Arijeet Roy Chowdhury in his mathematical journey. The second  author would like to thank  CSIR, Govt. of India, for the financial support in the form of Senior Research Fellowship under the mentorship of Prof. Kallol Paul.} 
	\maketitle

	\section{introduction}

The theory of $k$-smoothness plays an important role in understanding the geometry of Banach spaces. In recent years the study has gained momentum and the $k$-smoothness of bounded linear operators between  Hilbert spaces has been completely characterized in \cite{MDP20,W2}. However,  the $k$-smoothness of operators between  Banach spaces seems much more involved and it is still far from being well-understood. It should be noted that some progress has been made in this direction, for which we refer the readers to \cite{DMP,KS, LR, MP, MDP20, MPS}. In this article, we continue the study of the $k$-smoothness of operators between arbitrary Banach spaces. This gives us a better insight into the geometric structures of the unit sphere of  operator spaces. Before proceeding further, let us fix the relevant notations and terminologies.\\

We use the symbols $ \mathbb{X}, \mathbb{Y}$ to denote Banach spaces  over the field  $\mathbb{K},$ where  $\mathbb{K}$ is either the complex field $\mathbb{C}$ or the real field $\mathbb{R}.$  Let $ B_{\mathbb{X}}= \{ x \in \mathbb{X}: \|x\| \leq 1\} $  and $ S_{\mathbb{X}}= \{ x \in \mathbb{X}: \|x\| = 1\} $ denote  the unit ball and  unit sphere of $\mathbb{X},$ respectively. The  dual space of $ \mathbb{X}$ is denoted by  $ \mathbb{X}^*.$  Let $ \mathbb{L}(\mathbb{X}, \mathbb{Y})$ and $ \mathbb{K}(\mathbb{X}, \mathbb{Y}) $ denote the Banach space of all bounded linear operators and compact linear operators from $\mathbb{X}$ to $\mathbb{Y},$ respectively.   If $\mathbb{X}= \mathbb{Y},$ then we write $\mathbb{L}(\mathbb{X}, \mathbb{Y})= \mathbb{L}(\mathbb{X})$  and  $ \mathbb{K}(\mathbb{X}, \mathbb{Y}) = \mathbb{K}(\mathbb{X}).$ 

 For a  non-empty convex subset $A$ of $  \mathbb{X},$ an element $ z \in A$ is said to be an extreme point of $A$, whenever  $z = (1-t)x + ty,$ for some $t \in (0, 1)$ and $ x, y \in A,$ implies that $ x = y = z.$ 
The set of all extreme points of a convex set $A$ is denoted by $Ext(A).$ The space $\mathbb{X}$ is said to be strictly convex if $Ext(B_{\mathbb{X}})= S_{\mathbb{X}},$ i.e., if every element of the unit sphere is an extreme point of the unit ball. A finite-dimensional space $\mathbb{X}$ is said to be a polyhedral if $B_{\mathbb{X}}$ is a polyhedron i.e., if $Ext(B_{\mathbb{X}})$ is finite.
An operator $T \in \mathbb{L}(\mathbb{X}, \mathbb{Y})$ is said to be an extreme contraction if $T$ is an extreme point of $ B_{\mathbb{L}(\mathbb{X}, \mathbb{Y})}.$ For $T\in \mathbb{L}(\mathbb{X},\mathbb{Y}),~M_T$ denotes the collection of all unit vectors of $\mathbb{X}$ at which $T$ attains its norm, i.e., $M_T=\{x\in S_\mathbb{X}:\|Tx\|=\|T\|\}$.  Also, the  kernel of $T,$ denoted by $ker~T,$ is defined as $ker~T:=\{ x \in \mathbb{X} : Tx= 0 \in \mathbb{Y}\}.$  An element $ f \in S_{\mathbb{X}^*}$ is said to be a support functional at $0 \neq x \in \mathbb{X}$ if $f(x)= \|x\|.$ The set of all support functionals at $x \in S_{\mathbb{X}},$  denoted by $J(x),$ is defined as $J(x)= \{ f \in S_{\mathbb{X}^*}: f(x)= 1\}.$ It is easy to observe that $J(x)$ is a closed weak$^*$- compact  convex subset of $\mathbb{X}^*.$ An element $x \in S_\mathbb{X}$ is said to be smooth if $J(x)$ is singleton. The space $\mathbb{X}$  is said to be smooth if each $x \in S_{\mathbb{X}}$ is smooth.  
We say that $x \in S_{\mathbb{X}}$ is $k$-smooth if $dim~span ~J(x)=k.$  We say that the order of smoothness of $x \in S_{\mathbb{X}}$ is $k,$ if $x$ is $k$-smooth.
The notion of $k$-smoothness or multi-smoothness was introduced by Khalil and Saleh  \cite{KS} as a generalized notion of smoothness. 
Our aim in this article is to study the $k$-smoothness of an operator $T$  in $ \mathbb{L}(\mathbb{X}, \mathbb{Y}).$  As we will see, the roles of  $J(T)$ and  $Ext(J(T))$ are very important in the whole scheme of things. The following lemma is useful in this connection and will be used repeatedly. Let us first recall that  a closed subspace $\mathbb{V}$ of $\mathbb{X}$ is said to be an $M$-ideal of $\mathbb{X}$ if $\mathbb{X}^*= \mathbb{V}^* \oplus_1 \mathbb{V}^\perp,$ where $\mathbb{V}^\perp = \{ x^* \in \mathbb{X}^*: \mathbb{V} \subset ker~x^*\}$ and if $x^*= x_1^*+ x_2^*$ is the unique decomposition of $x^*$ in $\mathbb{X}^*,$ then $\|x^*\|= \|x_1^*\|+ \|x_2^*\|.$

\begin{lemma}\cite[Lemma 3.1]{W}\label{wojcik}
	Suppose that $\mathbb{X}$ is a reflexive Banach space and $\mathbb{Y}$ is a Banach space. Also assume that $\mathbb{K}(\mathbb{X},\mathbb{Y})$ is an $M$-ideal in $\mathbb{L}(\mathbb{X},\mathbb{Y}).$ Let $T\in \mathbb{L}(\mathbb{X},\mathbb{Y}), \|T\|=1$ and  dist$(T,\mathbb{K}(\mathbb{X},\mathbb{Y}))<1.$  Then $M_T\cap Ext(B_\mathbb{X})\neq \emptyset$ and 
	\[Ext(J(T))=\{y^*\otimes x\in \mathbb{K}(\mathbb{X},\mathbb{Y})^*:x\in M_T\cap Ext(B_{\mathbb{X}}), y^*\in Ext ~J(Tx)\},\]
	where  $y^*\otimes x: \mathbb{K}(\mathbb{X},\mathbb{Y})\to \mathbb{K}$ is defined by $y^*\otimes x(S)=y^*(Sx)$ for every $S\in \mathbb{K}(\mathbb{X},\mathbb{Y}).$
\end{lemma}
For  more information on tensor products, the readers are referred to \cite{R}.

The concepts of Birkhoff-James orthogonality and Auerbach basis play an important role in our study. Let us now recall from pioneering article \cite{B, J} that  an element $x \in \mathbb{X}$ is said to be Birkhoff-James orthogonal to $y \in \mathbb{X}$ if $ \| x + \lambda y \| \geq \|x\|$ for all $\lambda \in \mathbb{K}.$  Symbolically, it is written as $ x \perp_B y.$  
	A Schauder basis $\mathcal{B}$ of a Banach space $ \mathbb{X}$ is called an Auerbach basis if for any $ x \in \mathcal{B},$ $\| x\|=1$ and $ x \perp_B span \{ \mathcal{B} \setminus \{ x\}\}.$
	Also  the basis  $\mathcal{B}$ of $\mathbb{X}$ is said to be a strong Auerbach basis of $\mathbb{X}$  if for any $\mathcal{C} \subset \mathcal{B},$ $span ~\mathcal{C}  \perp_B span \{\mathcal{B} \setminus \mathcal{C} \} .$ In the Banach space $\ell_p^n(\mathbb{R})$, it is easy to see that $ \{ e_1, e_2, \ldots, e_n\}$ is a strong Auerbach basis, where $ e_i = (0, \ldots, 0, 1, 0,\ldots,  0)$ with $1$ in $i$-th position.

Let 
$((  \alpha_i \beta_j))_{1 \leq i \leq n, 1 \leq j \leq p }$   denote the $np$-tuple of scalars 
$$ ( \alpha_1 \beta_1 , \alpha_1 \beta_2, \ldots, \alpha_1 \beta_p, \alpha_2 \beta_1, \alpha_2 \beta_2, \ldots, \alpha_2 \beta_p, \ldots, \alpha_n \beta_1, \alpha_n \beta_2,  \ldots, \alpha_n \beta_p),$$ 
which is an element of $\mathbb{K}^{np}.$ Moreover,  $\widetilde{e_{ij}}= ( 0, 0 \ldots, 0,1, 0, \ldots, 0,0) \in \mathbb{K}^{np}$ where $1$ is in $ij$-th position. Clearly $\{ \widetilde{e_{ij}}: 1 \leq i \leq n, 1 \leq j \leq p\}$ is the standard ordered basis of $\mathbb{K}^{np}.$ To study the $k$-smoothness of an operator, we introduce  the following definition.

\begin{definition}\label{definition:index}
	Let $\mathbb{X}, \mathbb{Y}$ be  Banach spaces and let $T \in S_{\mathbb{L}(\mathbb{X}, \mathbb{Y})}.$  Let $R$ be  a subset of $S_\mathbb{X}$ such that $R \cap Ext(B_{\mathbb{X}}) \neq \emptyset.$ Let  $ \{ v_1, v_2, \ldots, v_n \}$ be a  basis of $ span~R.$
	Suppose that   $\mathbb{W} =span \{ y^* \in  J(Tv), v \in  R \cap Ext(B_{\mathbb{X}})\}$ is a finite-dimensional subspace of $\mathbb{Y}^*$ and  let $\{ y^*_1, y^*_2, \ldots, y^*_p\}$ be  a basis of $\mathbb{W}.$ Then the index of smoothness of $T$ with respect to $R,$
	denoted by $ i_{R}(T),$ is defined as the dimension of $Z,$ where 
	$ Z=  span \bigg\{(( \alpha_i \beta_j))_{ 1 \leq i \leq n, 1 \leq j \leq p} \in \mathbb{K}^{pn}:  \sum_{i=1}^n \alpha_i v_i \in R \cap Ext(B_{\mathbb{X}}),  \sum_{j=1}^{p}\beta_j y^*_j \in Ext(J(\sum_{i=1}^{n} \alpha_i T v_i))
	\bigg\}.$
\end{definition}

Later,  we will show  that the index of smoothness of $T$ with respect to $R$   depends neither on the choice of basis of $span ~R$ nor on  the basis of $\mathbb{W.}$ The present article is divided into three sections including the introductory one. In  section I, we characterize the $k$-smoothness of an operator in  terms of our newly introduced index of smoothness. We also characterize  the $k$-smoothness of an operator on a strictly convex, smooth Banach space whose norm attainment set is the unit sphere of a subspace. In section II, we study the $k$-smoothness of operators between polyhedral Banach spaces. Here we also observe some interesting properties of $rank~1$ operators, which provide  better insights into the geometry of operator space.  Surprisingly, we will see that the $k$-smoothness of a rank $1$ operator between finite-dimensional polyhedral Banach spaces depends heavily on the dimensions of the spaces rather than the geometry. We end this section by recalling the definition of coproximinal subspace, which arises naturally in our study.

	\begin{definition}\cite{FF, PS}
		Let $\mathbb{X}$ be a Banach space and let $\mathbb{Y}$ be a subspace of $\mathbb{X}$. Given any $x \in \mathbb{X},$ we say that $y_0 \in \mathbb{Y}$ is a best coapproximation to $x$ out of $\mathbb{Y}$ if $\|y_0-y\| \leq \| x-y\|,$
		for all $y \in \mathbb{Y}.$  A subspace $\mathbb{Y}$ of the Banach space $\mathbb{X}$ is said to
		be coproximinal if a best coapproximation to any element of $\mathbb{X}$ out of $\mathbb{Y}$ exists.
	\end{definition}
	 
	 It should be noted that $y_0 \in \mathbb{Y}$ is a best coapproximation to $x \in \mathbb{X}$ out of $\mathbb{Y}$ if and only if $\mathbb{Y} \perp_B (x- y_0).$
	 
	\section{main result}

\section*{Section-I}

We begin this section with a result that characterizes  extreme support functionals of an operator, under some additional conditions. For this purpose, we require the following fundamental characterization of Birkhoff-James orthogonality by James \cite{J}. 

\begin{lemma}\cite{J}\label{J}
Let $\mathbb{X}$ be a Banach space and let $ x, y \in \mathbb{X}.$ Then $x \perp_B y$ if and only if there exists $f \in S_{\mathbb{X}^*}$ such that $f(x)=\|x\|$ and $ f(y)=0. $
\end{lemma}

\begin{theorem}\label{BhatiaSemrl1}
	Let $\mathbb{X}$ be a reflexive strictly convex Banach space and let $\mathbb{Y}$ be a smooth Banach space. Suppose  that $T \in S_{\mathbb{L}(\mathbb{X}, \mathbb{Y})}$  and $\mathbb{K}(\mathbb{X}, \mathbb{Y})$ is an $M$-ideal of $\mathbb{L}(\mathbb{X}, \mathbb{Y})$   with  $dist(T, \mathbb{K}(\mathbb{X}, \mathbb{Y})) < 1.$ Then
   $g \in J(T)$ is an element of $Ext(J(T))$ if and only if   there exists a fixed element  $x_0 \in M_T$ such that $Tx_0 \perp_B Ax_0,$ for every $A \in ker~g.$  
	\end{theorem}

\begin{proof}
	Observe that by the strict convexity of $\mathbb{X},$ $ M_T \cap Ext(B_{\mathbb{X}}) = M_T$ and by the smoothness of $\mathbb{Y},$ we get that $J(y) $ is a singleton set for each $ y\in \mathbb{Y}.$  Using Lemma \ref{wojcik}, we obtain 
 \[ Ext\,(J(T))= \{ y^* \otimes x: x \in M_T, \{y^*\} =     J(Tx)  \}. \]
  We first prove the sufficient part of the theorem. Suppose that $ g \in J(T)$  is such that there exists a fixed element $x_0 \in M_T$  satisfying $Tx_0 \perp_B Ax_0,$ for every  $ A \in ker\,g.$  Since $Tx_0 \perp_B Ax_0$ and $\mathbb{Y}$ is smooth,  using Lemma \ref{J} we get a bounded linear functional  $y_0^*$ on $\mathbb{Y}$ such that $y_0^*(Tx_0) = \|Tx_0\|=1$ and $y_0^*(Ax_0) =0.$  Thus $y_0^*\otimes x_0(T) = \|T\| = g(T) $ and $y_0^* \otimes x_0 (A) =0$ for every $A \in ker\,g.$  This shows that $g = y_0^* \otimes x_0,$ where $ x_0 \in M_T$ and $J(Tx_0) = \{y_0^*\}$ and so $g \in Ext\,(J(T)).$ For the necessary part, we suppose that  $g \in Ext\,(J(T)).$ So, $g$ can be written as $g = y_0^* \otimes x_0,$
 for some $x_0\in M_T$ and $\{y_0^* \}=J(Tx).$ Then  $ y_0^*(Tx_0) = \|Tx_0\|$ and $y_0^* (Ax_0)= y_0^* \otimes x_0 (A)= g(A)= 0,$  for every $ A \in ker \,g.$ Using Lemma \ref{J}, we conclude that  $Tx_0 \perp_B A x_0,$ for every $A \in ker \,g.$
\end{proof}

The above theorem fulfills the purpose of characterizing the extreme support functionals of a bounded linear operator in a special setting.  Moreover, it provides two strengthened versions of  the Bhatia-\v Semrl Theorem \cite[Th. 1.1]{BS} on Banach spaces, under these conditions. Let us first recall the Bhatia-\v Semrl Theorem.

\begin{theorem}(Bhatia-\v Semrl Theorem)
		Let $ \mathbb{H}$ be a finite-dimensional Hilbert space and let  $T, A \in \mathbb{L}(\mathbb{H}).$ Then $T \perp_B A$ if and only if there exists $x \in  M_T$ such	that $\langle Tx, Ax \rangle =0.$
	\end{theorem}

Although the Bhatia- \v Semrl Theorem has been presented in the setting of Hilbert spaces, it has also been studied  in the Banach space setting by many authors. In \cite[Th. 2.1]{SP}, the authors generalized the above theorem for linear operators on finite-dimensional real Banach spaces:
 
\begin{theorem} \cite[Th. 2.1]{SP}
 Let $\mathbb{X}$ be a finite-dimensional real Banach space. Let $T \in \mathbb{L}(\mathbb{X})$ be such that $M_T = \pm D$, where $D$ is a closed, connected subset of $S_{\mathbb{X}}$. Then for $A \in \mathbb{L}(\mathbb{X})$ with $T \perp_B A$, there exists $x \in D$ such that $Tx \perp_B Ax$.
 \end{theorem}
 The Bhatia-\v Semrl Theorem in the setting of complex Banach spaces has been studied in \cite[Th. ~4.3]{RBS}. We refer  the readers to \cite{S, SP, SRP} and the references therein for further  reading in this context.  A bounded linear operator $T$ is said to satisfy Bhatia-\v Semrl  (B\v S) property \cite{SPH} if for each bounded linear operator $A$, $T \perp_B A$ implies there exists $x \in M_T$ ($x$ depending on $A$)  such that $ Tx \perp_B Ax.$   Investigations of operators  satisfying B\v S property have been conducted in \cite{ KL, RSDP, SRDP}. We next present the promised strengthening of the Bhatia-\v Semrl Theorem, under certain additional conditions.

\begin{theorem}\label{BhatiaSemrl2}
 Let $\mathbb{X}$ be a reflexive strictly convex Banach space and let $\mathbb{Y}$ be a smooth Banach space. Suppose  that $T \in S_{\mathbb{L}(\mathbb{X}, \mathbb{Y})}$  and $\mathbb{K}(\mathbb{X}, \mathbb{Y})$ is an $M$-ideal of $\mathbb{L}(\mathbb{X}, \mathbb{Y})$   with  $dist(T, \mathbb{K}(\mathbb{X}, \mathbb{Y})) < 1.$ Then
 \begin{itemize}
 	\item [(i)] for a subspace  $\mathbb{W}$ of $\mathbb{L}(\mathbb{X}, \mathbb{Y}),$  there exists a fixed $x_0 \in M_T$ such that $Tx_0 \perp_B Ax_0,$ for any $A \in \mathbb{W}$ if and only if $\mathbb{W} \subset ker~g,$ for some $g \in Ext(J(T)).$ 
 	\item [(ii)]  a subspace  $\mathbb{W}$ of $\mathbb{L}(\mathbb{X}, \mathbb{Y})$ satisfies $Tx \perp_B Ax,$ for every  $x \in M_T$ and for every $A \in \mathbb{W}$ if and only if $\mathbb{W} \subset \bigcap_{ g \in Ext(J(T))}ker~g.$ 
 \end{itemize}
  \end{theorem}

\begin{proof}
	Since $\mathbb{X}$ is strictly convex and $\mathbb{Y}$ is smooth, using Lemma \ref{wojcik}, we obtain that
	\[	Ext(J(T))= \bigg\{ y^* \otimes x: x \in M_T, \{y^*\} =     J(Tx)  \bigg\}.	\]
	
(i) The sufficient part of the proof is immediate from Theorem \ref{BhatiaSemrl1}, so we only  prove the necessary part. Suppose that $J(Tx_0)= \{y_0^*\}$ and $g = y_0^* \otimes x_0.$ Since for any $A \in \mathbb{W},$ $Tx_0 \perp_B Ax_0,$ so using Lemma \ref{J}, we get $y_0^*(Ax_0)=0.$ Therefore, for  any $A \in \mathbb{W},$ $g(A)= y_0^* \otimes x_0 (A)= y_0^* (Ax_0) =0,$ which implies that $\mathbb{W} \subset ker~g$ and the fact that $g \in Ext (J(T))$ completes the theorem. \\
 
 (ii) The necessary part follows easily from the proof of (i).  For the converse part, let  $\mathbb{W}$ be a subspace of $\mathbb{L}(\mathbb{X}, \mathbb{Y})$ such that $\mathbb{W} \subset \bigcap_{ g \in Ext(J(T))}ker~g.$    Let $ x \in M_T$ be arbitrary but fixed after choice. Consider $J(Tx)= \{ y^*\} .$  Then $y^* \otimes x$ is an extreme supporting functional at $T.$ So for each $ A \in \mathbb{W},$ we get that $ y^* \otimes x (A) =0 .$  Thus 
 $ y^*(Ax) = y^* \otimes x (A) = 0 $ and $ y^*(Tx) =  y^* \otimes x (T) = \|T\|=1,$ and so by applying Lemma \ref{J}, we conclude that $Tx \perp_B Ax. $  
	\end{proof}
  
\begin{remark}
It is worth mentioning that characterization of $T \perp_B A$ in the Hilbert space  setting was also investigated in \cite{P, PHD}. For a more general study of the same problem, in the setting of complex Banach spaces, the readers are referred to \cite{RBS}. In this connection, we also note that the conjecture posed in \cite{SPH} still remains an open question. The conjecture is stated as : 

\emph{A linear operator $T$ on an $n$-dimensional real normed linear space $\mathbb{X}$
	satisfies the B\v S property if and only if the set of unit vectors on which $T$ attains norm is connected in the corresponding projective space $\mathbb{R} P^{n-1} \equiv S_{\mathbb{X}} \setminus \{x \sim -x \}.$}
\end{remark}

Going back to the study of $k$-smoothness of operators, we recall the following well known result.

\begin{prop}\label{prop1}
	Suppose that $\mathbb{X}, \mathbb{Y}$ are Banach spaces. If $\{ x_1, x_2, \ldots, x_n\}$ is a linearly independent subset of $ \mathbb{X}$ and $ \{ y_1^*, y_2^*, \ldots, y_p^*\}$ is a linearly independent subset of $\mathbb{Y}^*$ then $\{ y_i^* \otimes x_j: 1 \leq i \leq p, 1 \leq j \leq n \}$ is a linearly independent subset of $\mathbb{L}(\mathbb{X}, \mathbb{Y})^*.$ \\
\end{prop}

In the next proposition we give an estimation of the number of linearly independent elements in $Ext (J(T))$ for a bounded linear operator on a Banach space, not necessarily finite-dimensional. See \cite[Th. 2.2]{MP} for an  analogous result on finite-dimensional Banach spaces.

\begin{prop}\label{minimum}
	Let $\mathbb{X}$ be a reflexive Banach space and let $\mathbb{Y}$ be any Banach space. Suppose  that $T \in S_{\mathbb{L}(\mathbb{X}, \mathbb{Y})}$  and $\mathbb{K}(\mathbb{X}, \mathbb{Y})$ is an $M$-ideal of $\mathbb{L}(\mathbb{X}, \mathbb{Y})$   with  $dist(T, \mathbb{K}(\mathbb{X}, \mathbb{Y})) < 1.$ Suppose that $\{ x_1, x_2, \ldots, x_r\}$ is a maximal linearly independent set in $M_T \cap Ext(B_{\mathbb{X}})$ and $Tx_i$ is $m_i$-smooth, for $ 1 \leq i \leq r.$ Then $dim~span (J(T)) \geq \sum_{i=1}^r m_i.$
\end{prop}

\begin{proof}
	Suppose that $ \{y^*_{i1}, y^*_{i2}, \ldots, y^*_{im_i}\}$ is a linearly independent set in $Ext(J(Tx_i)),$ for any $i, 1 \leq i \leq r.$ Let $H = span~ M_T.$ Clearly, $H$ is a finite-dimensional  subspace of $\mathbb{X}.$ 
	From Lemma \ref{wojcik} we get, 
	\[
	Ext(J(T))= \bigg\{ y^* \otimes x: x \in M_T \cap Ext(B_{\mathbb{X}}), y^* \in Ext( J(Tx))  \bigg\}.
	\]
	Since $ x_i \in M_T \cap Ext(B_{\mathbb{X}})$ and $y^*_{ij} \in Ext(J(Tx_i)),$  then $y^*_{ij} \otimes x_i \in Ext(J(T)),$ for any $ 1\leq i \leq r, 1 \leq j \leq m_i.$  Clearly, $R = \{ y^*_{ij} \otimes x_i: 1\leq i \leq r, 1 \leq j \leq m_i\}$ is a subset of the space  $ \mathbb{Y}^* \otimes H .$ To prove the theorem it is enough to show that the set $R$ is linearly independent. Let $\sum_{i=1}^{r} \sum_{j=1}^{m_i} c_{ij} y^*_{ij} \otimes x_i=0,$ where $c_{ij} \in \mathbb{K},$ for any $1 \leq i \leq r, 1 \leq j \leq m_i.$ Let $\mathcal{B}= \{ e_\alpha: \alpha \in \Lambda\}$ be a Hamel basis of $\mathbb{Y}.$   For any $1 \leq t \leq r, \alpha \in \Lambda,$ we define $A_{t\alpha}: H \rightarrow \mathbb{Y} $ as follows:
	\[
	A_{t\alpha}(x_t)= e_{\alpha} \quad \quad 	A_{t\alpha}(x_i)=0,~ \forall ~ i \in \{ 1 , 2, \ldots, r\} \setminus \{t\}.
	\] 
	Since $H$ is finite-dimensional, clearly $A_{t \alpha}  $ is bounded.
	Then for any $1 \leq t \leq r$ and for any $\alpha \in \Lambda,$ $ \sum_{i=1}^{r} \sum_{j=1}^{m_i} c_{ij} y^*_{ij} \otimes x_i (A_{t\alpha})= 0 \implies \sum_{i=1}^{r}\sum_{j=1}^{m_i} c_{ij} y^*_{ij} (A_{t\alpha}(x_i))=0 \implies \sum_{j=1}^{m_t} c_{tj} y^*_{tj} (A_{t\alpha}(x_t))=0 \implies \sum_{j=1}^{m_t} c_{tj} y^*_{tj}(e_\alpha) =0. $  Therefore, for any $1 \leq t \leq r,$ $\sum_{j=1}^{m_t} c_{tj} y^*_{tj}=0 \implies c_{tj}=0,$ for any $1 \leq j \leq m_t.$ Consequently, $R$ is  linearly independent and this completes the theorem.
\end{proof}

	In the following theorem we characterize the $k$-smoothness of rank $1$ operators.

\begin{theorem}\label{rank1}
		Let $\mathbb{X}$ be a reflexive Banach space and let $\mathbb{Y}$ be any Banach space. Suppose that $T \in S_{\mathbb{L}(\mathbb{X}, \mathbb{Y})}$  is of rank $1.$  Let $\mathbb{K}(\mathbb{X}, \mathbb{Y})$ be an $M$-ideal in $\mathbb{L}(\mathbb{X}, \mathbb{Y})$    and $dist(T, \mathbb{K}(\mathbb{X}, \mathbb{Y})) < 1.$ 
		Let the cardinality of a maximal linearly independent set of $M_T \cap Ext(B_{\mathbb{X}})$ be $n$ and  let $T(M_T)= \{\lambda y : | \lambda | =1 \},$ for some $y \in S_{\mathbb{Y}}.$  Then $T$ is $k$-smooth if and only if $y  $ is $m$-smooth and $mn=k.$
\end{theorem}

\begin{proof}
	Let $\{ x_1, x_2, \ldots, x_n\}$  be a maximal linearly independent set of $M_T \cap Ext(B_{\mathbb{X}}).$ 
	Observe that   $y^* \in J(y)$ if and only if $\frac{1}{\lambda} y^* \in J(\lambda y),$ whenever $|\lambda|=1.$ Also  $T(M_T)= \{\lambda y : \, | \lambda |  =1 \}.$  So using Lemma \ref{wojcik}, we get
	\begin{eqnarray*}
	span~	Ext(J(T)) &=& span~ \bigg\{ y^* \otimes x: x \in M_T \cap Ext(B_{\mathbb{X}}), y^* \in Ext( J(Tx))  \bigg\} \\
	&=& span~\bigg\{ y^* \otimes x: x \in M_T \cap Ext(B_{\mathbb{X}}), y^* \in Ext( J(y))  \bigg\}. 
\end{eqnarray*}
Let us now prove the sufficient part of the theorem.	Suppose that $y^*_1, y^*_2, \ldots, y^*_m \in S_{\mathbb{Y}^*} $ are  linearly independent extreme support functionals of $y \in S_{\mathbb{Y}}.$ Observe that $y^*_j \otimes x_i \in Ext(J(T)),$ for each $i, j $ where $1 \leq i \leq n $ and $ 1 \leq j \leq m.$ Moreover, $\{ y^*_j \otimes x_i: 1 \leq j \leq m, 1 \leq i \leq n\}$ is linearly independent. Let  $x \in M_T \cap Ext(B_{\mathbb{X}}),~y^* \in Ext(J(y))$ and suppose that $x = \sum_{i=1}^{n} c_i x_i, y^*= \sum_{j=1}^{m} d_j y^*_j,$ for some $c_i, d_j \in \mathbb{K}.$ Then $y^* \otimes x \in Ext(J(T)).$ Now, $y^* \otimes x= \sum_{j=1}^{m} d_j y^*_j \otimes \sum_{i=1}^{n} c_i x_i = \sum_{j=1}^{m} \sum_{i=1}^{n} d_j c_i y^*_j \otimes x_i \in span \{ y^*_j \otimes x_i: 1 \leq j \leq m, 1 \leq i \leq n\}.$  Therefore $ dim ~span~ Ext(J(T))= dim ~ span~ \{ y^*_j \otimes x_i : 1 \leq i \leq n, 1 \leq j \leq m\}= mn.$  This completes the proof.

We next prove the necessary part. First of all we assume that $dim~span J(y)$ is not finite and let $\{y_{\alpha}^*: \alpha \in \Lambda\}$ is a linearly independent set in  $span J(y).$ Then following the similar arguments given in the sufficient part we show that $\{ y_{\alpha}^* \otimes x_i: \alpha \in \Lambda, 1 \leq i \leq n\}$ is a linearly independent set in $span~J(y).$ That contradicts that order of smoothness of $T$ is finite. Now suppose that $dim~span J(y)=m.$ Again following the similar arguments given in the sufficient part we can show that $k=mn.$ 
\end{proof}

Using Theorem \ref{rank1} and \cite[Th. 4.2]{W2}, we get the following corollary.

\begin{cor}
	Let $\mathbb{X},\mathbb{Y}$ be finite-dimensional Banach spaces such that $dim ~\mathbb{X}=n, \, dim ~\mathbb{Y}=m.$ Suppose that $T \in S_{\mathbb{L}(\mathbb{X}, \mathbb{Y})}$ is an operator of rank $1$ such that  $T(M_T)= \{\lambda  y : | \lambda|= 1 \},$ for some $y \in S_{\mathbb{Y}}.$  If $dim~span(M_T)=n$ and $y$ is $m$-smooth, then $T$ is an extreme contraction. \\
\end{cor}

Next we show that the newly introduced  notion of index of smoothness of an operator with respect to a set is basis invariant.

\begin{prop}\label{invariance}
	Let $\mathbb{X}, \mathbb{Y}$ be  Banach spaces and let  $T \in S_{\mathbb{L}(\mathbb{X}, \mathbb{Y})}$.  Let $R$ be  a subset of $S_\mathbb{X}$ such that $R \cap Ext(B_{\mathbb{X}}) \neq \emptyset.$      Suppose that  $\mathbb{W}=span \{ f \in J(Tv): v \in R \cap Ext(B_{\mathbb{X}}) \}$ is a finite-dimensional subspace of $\mathbb{Y}^*.$
	Then the index  of smoothness of $T$ with respect to $R$  is independent of the choice of basis of $span~R$ as well as the basis of $\mathbb{W}.$
\end{prop}

\begin{proof}
  Suppose that $\{y^*_1, y^*_2, \ldots, y^*_m\}$ is a basis of $\mathbb{W}$ and $\{ x_1, x_2, \ldots, x_n\}$ is a basis of $span~R.$ Suppose that $i_R(T) = k$ with respect to these bases.  Let  $ Z = span \bigg\{((\alpha_i \beta_j ))_{1 \leq i \leq n, 1 \leq j \leq m} \in \mathbb{K}^{mn}:  \sum_{i=1}^n \alpha_i x_i \in R \cap Ext(B_{\mathbb{X}}),  \sum_{j=1}^{m}\beta_j y^*_j \in Ext(J(\sum_{i=1}^{n} \alpha_i Tx_i))	\bigg\}.$
	Assume that $ B=\{ ((\alpha^l_i \beta^l_j ))_{1 \leq i \leq n, 1 \leq j \leq m} : 1 \leq l \leq k\}$ is a basis of $Z.$ Let $\{ x'_1, x'_2, \ldots, x'_n\}$ be  another basis of $span~R$ and let $ x_i = \sum_{j=1}^{n} c_{ij} x'_j,$ for some scalars $c_{ij}.$  Clearly, $C= (c_{ji})_{1 \leq i, j \leq n}$ is an invertible matrix. It is easy to observe that whenever $\sum_{i=1}^{n} \alpha_i x'_i= \sum_{i=1}^{n} \gamma_i x_i,$ then $
	(\alpha_1, \alpha_2, \ldots, \alpha_n)^T = C (\gamma_1, \gamma_2, \ldots, \gamma_n)^T  $ ( here $A^T$ is the transpose of a matrix $A$). 
Consider $$P= \begin{pmatrix}
		C & 0 & 0 & \ldots 0 \\
		0 & C & 0 & \ldots 0 \\
		. & . & .\\
		. & . & .\\
		. & . & .\\
		0 &0 & 0 & \ldots C
	\end{pmatrix}_{nm \times nm}.$$ 
	We observe that for any $ 1\leq l \leq k, $ $( ( \gamma^l_i \beta^l_j))_{1 \leq i \leq n, 1 \leq j \leq m}= P ( ( \alpha^l_i  \beta^l_j))_{1 \leq i \leq n, 1 \leq j \leq m}.$ Since, $B$ is linearly independent and $P$ is invertible, then $\{( ( \gamma^l_i  \beta^l_j))_{1 \leq i \leq n, 1 \leq j \leq p}: 1 \leq l \leq k \}$ is linearly independent.  Let    $i_R(T) = k'$ with respect  to the bases  $\{y^*_1, y^*_2, \ldots, y^*_m\}$  and $\{ x'_1, x'_2, \ldots, x'_n\}.$  Thus we get $k' \geq k $ and proceeding similarly we can show that $k \geq k'.$  Thus $i_R(T) $ is invariant with respect to the basis of $span\,R.$ In the same fashion we can show that $i_R(T) $ is invariant with respect to the basis of $\mathbb{W}.$ 
\end{proof}

We are now in a position to characterize the $k$-smoothness of an operator in terms of the index of smoothness.

	\begin{theorem}\label{order}
	Let $\mathbb{X}$ be a reflexive Banach space  and let  $\mathbb{Y}$ be any Banach space. Suppose that $T \in S_{\mathbb{L}(\mathbb{X}, \mathbb{Y})}$  and  $\mathbb{K}(\mathbb{X}, \mathbb{Y})$ is an $M$-ideal in $\mathbb{L}(\mathbb{X}, \mathbb{Y})$  with  $dist(T, \mathbb{K}(\mathbb{X}, \mathbb{Y}))< 1.$  Suppose that $dim~span {M_T}=n$  and   $span \{ g \in J(Tx), x \in M_T \cap Ext(B_{\mathbb{X}})\}$ is a finite-dimensional subspace of $\mathbb{X}^*.$Then $T$ is $k$-smooth if and only if   $ i_{M_T}(T) = k.$ 
\end{theorem}

\begin{proof}
	Let $   \{ x_1, x_2, \ldots, x_n\}$ be a maximal linearly independent set of $M_T.$
	 Suppose that $\{ y^*_1, y^*_2, \ldots, y^*_m\}$ is a basis of $span \{ y^* \in J(Tx), x\in M_T \cap Ext(B_{\mathbb{X}})\}.$
	 Using Lemma \ref{wojcik} we get, 
	 	\begin{eqnarray*}
	 		& & Ext(J(T))   \\
	 & = & \bigg\{ y^* \otimes x: x \in M_T \cap Ext(B_{\mathbb{X}}), y^* \in Ext( J(Tx))  \bigg\}\\
	& = &    \bigg\{ \sum_{j=1}^{m} \beta_j y^*_j \otimes  \sum_{i=1}^{n} \alpha_i x_i:  \sum_{i=1}^{n}\alpha_i x_i \in M_T \cap Ext(B_{\mathbb{X}}), \sum_{j=1}^{m} \beta_j y^*_j  \in Ext\Big( J(\sum_{i=1}^{n}\alpha_i Tx_i)\Big)  \bigg\}\\
		& = &    \bigg\{ \sum_{j=1}^{m} \sum_{i=1}^{n} \alpha_i \beta_j y^*_j \otimes x_i:  \sum_{i=1}^{n}\alpha_i x_i \in M_T \cap Ext(B_{\mathbb{X}}), \sum_{j=1}^{m} \beta_j y^*_j  \in Ext\Big( J(\sum_{i=1}^{n}\alpha_i Tx_i)\Big)  \bigg\}
	\end{eqnarray*}
	Suppose that $
W_1= span ~Ext(J(T)) $
and 
$
W_2 = span \bigg\{(( \alpha_i \beta_j)) _{1 \leq i \leq n, 1 \leq j \leq m} \in \mathbb{K}^{mn}:  \sum_{i=1}^n \alpha_i x_i \in M_T \cap Ext(B_{\mathbb{X}}), \sum_{j=1}^{m}\beta_j y^*_j \in Ext(J(\sum_{i=1}^{n} \alpha_i Tx_i))
\bigg\}.
$
To prove the theorem we only need to show that $dim ~W_1= dim ~W_2.$ Suppose that $dim ~ W_1=k$ and let $S= \{ \sum_{j=1}^{m} \sum_{i=1}^{n}  \alpha_{1i} \beta_{1j} y^*_j \otimes x_i,  \sum_{j=1}^{m} \sum_{i=1}^{n}  \alpha_{2i} \beta_{2j} y^*_j \otimes x_i, \ldots, $ $ \sum_{j=1}^{m} \sum_{i=1}^{n}  \alpha_{ki} \beta_{kj} y^*_j \otimes x_i\}$ be a linearly independent set in $W_1.$
 Then clearly, $ \sum_{i=1}^{n}\alpha_{ti} x_i \in M_T \cap Ext(B_{\mathbb{X}})$ and $\sum_{j=1}^{m} \beta_{tj} y^*_j \in Ext( J( \sum_{i=1}^{n}\alpha_{ti} Tx_i)),$ for each  $t, \, 1 \leq t \leq k.$ 
Suppose that for some $c_1, c_2, \ldots, c_k \in \mathbb{K},$ $\sum_{t=1}^{k} c_t ((  \alpha_{ti} \beta_{tj} ))_{1 \leq i \leq n, 1 \leq j \leq m}= 0 \implies \sum_{t=1}^{k} c_t  \alpha_{ti} \beta_{tj}=0,$ for each $ i, \,1 \leq i \leq n $ and $ j, \,1 \leq j\leq m.$ Then for any $A \in \mathbb{L}(\mathbb{X}, \mathbb{Y}),$
 \[
\sum_{t=1}^{k} c_t \bigg( \sum_{j=1}^{m} \sum_{i=1}^{n} \alpha_{ti} \beta_{tj} y^*_j \otimes x_i \bigg) A= \sum_{j=1}^{m} \sum_{i=1}^{n} y^*_j \bigg(A \bigg(\sum_{t=1}^{k}  c_t  \alpha_{ti}\beta_{tj} x_i \bigg) \bigg)=0.
 \]
Since $S$ is a linearly independent  set in $W_1,$ we get  $c_t=0,$ for each  $t, \, 1 \leq t \leq k.$ Therefore, $\{ ((  \alpha_{ti} \beta_{tj} ))_{1 \leq i \leq n, 1 \leq j \leq m}: 1 \leq t \leq k\}$ is a linearly independent set in $W_2.$ So, $dim ~W_1=k \leq dim ~ W_2.$ Now assume that $dim ~W_2= r$ and $ \{ ((  \alpha_{ti} \beta_{tj} ))_{1 \leq i \leq n, 1 \leq j \leq m}: 1 \leq t \leq r\}$ is a basis of $W_2.$ Then it is immediate that 
\[
dim~span \{ (\overline{\alpha_{1i} \beta_{1j}}, \overline{\alpha_{2i} \beta_{2j}}, \ldots, \overline{\alpha_{ri} \beta_{rj}} ): 1 \leq i \leq n, 1 \leq j \leq m\}= r.
\]
Suppose that for some $d_1, d_2, \ldots, d_r \in \mathbb{K},$ $ \sum_{t=1}^{r} d_t (\sum_{j=1}^{m} \sum_{i=1}^{n}  \alpha_{ti} \beta_{tj} y^*_j \otimes x_i) = 0 \implies \sum_{j=1}^{m} \sum_{i=1}^{n} (\sum_{t=1}^{r} d_t  \alpha_{ti} \beta_{tj} ) y^*_j \otimes x_i =0.$ Using  Proposition \ref{prop1}, for any $ 1 \leq i \leq n, 1 \leq j \leq m,$ $\sum_{t=1}^{r} d_t  \alpha_{ti} \beta_{tj}=0 \implies \langle (d_1, d_2, \ldots, d_r), (\overline{\alpha_{1i} \beta_{1j}}, \overline{\alpha_{2i} \beta_{2j}}, \ldots, \overline{\alpha_{ri} \beta_{rj}}) \rangle = 0 . $ Since $dim~span \{ (\overline{\alpha_{1i} \beta_{1j}}, \overline{\alpha_{2i} \beta_{2j}}, \ldots, \overline{\alpha_{ri} \beta_{rj}} ): 1 \leq i \leq n, 1 \leq j \leq m\}= r,$ it is easy to verify that $(d_1, d_2, \ldots, d_r)= 0  \implies d_i=0,$ for any $1 \leq i \leq r.$ This implies that $\{ \sum_{j=1}^{m} \sum_{i=1}^{n}  \alpha_{ti} \beta_{tj} y^*_j \otimes x_i: 1 \leq t \leq r\}$ is a linearly independent set in $W_1.$ Thus $dim ~W_2=r \leq dim ~W_1.$ Therefore, $dim~W_1= dim~W_2,$ as required.
   
\end{proof}

The study of extreme contractions between Banach spaces is an active area of research. For some of the recent results in this direction, the readers are referred to \cite{MPD, RRBS, S1, SSP}. Using \cite[Th. 2.2]{MPD}, the next result follows immediately as a corollary to the above theorem, which characterizes the extreme contractions between polyhedral Banach spaces, in terms of the index of smoothness. 

\begin{cor}
	Let $\mathbb{X}, \mathbb{Y}$ be  finite-dimensional polyhedral Banach spaces and let $dim ~\mathbb{X}=m, dim ~\mathbb{Y}=n.$ Let  $T \in S_{\mathbb{L}(\mathbb{X}, \mathbb{Y})}.$ Then $T$ is an extreme contraction  if and only if $i_{M_T}(T) = mn.$
\end{cor}

Before we present the connection between $k$-smoothness of an operator with coproximinality, let us first characterize the coproximinal subspaces in smooth Banach space.	

	\begin{lemma}\label{lemma:coproximinal}
		Let $\mathbb{X}$ be a smooth Banach space and let $\mathbb{Y}$ be an $n$-dimensional subspace of $\mathbb{X}.$ Then $\mathbb{Y}$ is coproximinal if and only if $dim~span \{ y^* \in S_{\mathbb{X}^*}: y^* \in J(y), y \in \mathbb{Y}\}=n.$ 
	\end{lemma}
	
	\begin{proof}
		
		Let $S = \{ y^* \in S_{X^*}: y^* \in J(y), y \in \mathbb{Y}\}$ and let $W= \cap_{y^*\in S} ker ~y^*.$  Let us first observe that $\mathbb{Y} \cap W= \{0\}.$ If possible let, $ y (\neq 0) \in \mathbb{Y} \cap W$ and let $J(y)=\{y^*\},$ as $\mathbb{X}$ is smooth. Then $y^* \in S$ and since $y \in W,$ we obtain that $y^*(y)=0,$ a contradiction to the fact that $y^* \in J(y).$
		
		We first prove the necessary part.  Since $\mathbb{Y}$ is coproximinal, for any $x \in \mathbb{X}$ there exists $y_0 \in \mathbb{Y}$ such that $y_0$ is the best coapproximation to $x$ out of $\mathbb{Y},$  and so  $\mathbb{Y} \perp_B x-y_0.$ 
		Then for each $y \in \mathbb{Y}, \, y \perp_B x-y_0 .$ Since $\mathbb{X}$ is smooth,   for each $y^* \in S,$ there exists $y \in S_{\mathbb{Y}}$ such that $J(y)= \{y^*\}$ and therefore, using Lemma  \ref{J} we obtain $ y^*(x-y_0)=0 .$   Then  $x-y_0 \in W,$  and  $ x = (x -y_0) +y_0, $ where $x-y_0 \in W, y_0 \in \mathbb{Y}$ . Thus $\mathbb{X}= \mathbb{Y}+ W.$  As $\mathbb{Y} \cap W= \{0\},$ we have $\mathbb{X}= \mathbb{Y} \oplus W.$ This implies that $W$ is a subspace of codimension $n$ and consequently, $dim~span ~S=n.$
		
		\smallskip
			Now we prove the sufficient part.
		Assume that $dim~span~ S=n.$
	 Then $W$ is  a  subspace of $\mathbb{X}$ of codimension $n.$ 
		Since $dim~\mathbb{Y}=n,~dim~\sfrac{\mathbb{X}}{W}=n$ and $\mathbb{Y} \cap W= \{0\},$ therefore, $\mathbb{X}= \mathbb{Y} \oplus W.$ Therefore, for any $x \in \mathbb{X},$ there exists $y_0 \in \mathbb{Y}$ such that $x - y_0 \in W.$ Let $y \in \mathbb{Y}$ and suppose that $J(y)= \{y^*\},$ clearly, $y^* \in S$ and  $y^*(x-y_0)=0.$ Therefore, from Lemma \ref{J}, $y \perp_B x-y_0.$  As $y$ is chosen arbitrarily,  $ \mathbb{Y} \perp_B x -y_0,$ which implies that $y_0$ is a best coapproximation to $x$ out of $\mathbb{Y},$ and consequently, $\mathbb{Y}$ is coproximinal subspace of $\mathbb{X}.$
	\end{proof}

Using Lemma \ref{lemma:coproximinal} and Theorem \ref{order}, we get the following corollary.

	\begin{cor}
	Let $\mathbb{X}$ be a reflexive smooth Banach space  and let  $\mathbb{Y}$ be a Banach space. Suppose that $T \in S_{\mathbb{L}(\mathbb{X}, \mathbb{Y})}$  and  $\mathbb{K}(\mathbb{X}, \mathbb{Y})$ is an $M$-ideal in  $\mathbb{L}(\mathbb{X}, \mathbb{Y})$  with  $dist(T, \mathbb{K}(\mathbb{X}, \mathbb{Y}))< 1.$ Also assume that the cardinality of a maximal linearly independent subset of $M_T$ is $n$ and   $span \{T(M_T)\}$ is a coproximinal subspace of $\mathbb{Y}.$ Then $T$ is $k$-smooth if and only if $i_{M_T}(T) = k.$
\end{cor}

A complete characterization of $k$-smoothness for Hilbert space operators was given in  \cite[Th. 2.1]{MDP20}. Using the notions developed in this article, we extend the result substantially  for operators between Banach spaces.

	\begin{theorem}\label{smooth}
		Let $\mathbb{X}, \mathbb{Y}$ be strictly convex, smooth, reflexive Banach spaces.  Suppose  $T \in S_{\mathbb{L}(\mathbb{X}, \mathbb{Y})}$   and $\mathbb{K}(\mathbb{X}, \mathbb{Y})$ is an $M$-ideal in $\mathbb{L}(\mathbb{X}, \mathbb{Y})$with  $dist(T, \mathbb{K}(\mathbb{X}, \mathbb{Y})) < 1.$ Let $M_{T}$ be the unit sphere of an $n$-dimensional subspace  $\mathbb{X}_0$ of $\mathbb{X}.$ Suppose that  $T(\mathbb{X}_0)$ is a coproximinal subspace of $\mathbb{Y}$ and  $T(\mathbb{X}_0)$ has a strong Auerbach basis. Then
		\begin{itemize}
				\item[(i)] $T$ is at least $\binom{n+1}{2}$-smooth.
			\item[(ii)]  $T$ is $n^2$-smooth, when $\mathbb{X}, \mathbb{Y}$ are over complex field.
		\end{itemize}
		\end{theorem}

	\begin{proof} Observe that $T(\mathbb{X}_0)$ is an $n$-dimensional subspace of $\mathbb{Y}.$
		Suppose that $\mathcal{B}= \{ y_1, y_2, \ldots, y_n\}$ is a strong Auerbach basis of $T(\mathbb{X}_0).$  Let $T(x_i) = y_i,$ for each $i, \,1 \leq i \leq n.$  Clearly, $\{x_1, x_2, \ldots, x_n\}$ is linearly independent.
	Let $W= span \{ y^* \in S_{\mathbb{Y}^*}: y^* \in J(y), y \in T(\mathbb{X}_0)\}.$  Since $T(\mathbb{X}_0)$ is coproximinal subspace in $\mathbb{Y},$ using  Lemma \ref{lemma:coproximinal},  we infer that $W$ is an $n$-dimensional subspace of $\mathbb{Y}^*.$ Since $\mathcal{B}$ is a strong Auerbach basis of $T(\mathbb{X}_0),$ there exists $y^*_i \in S_{\mathbb{Y}^*}$ such that $y^*_i(y_i)=1$ and $y^*_i(y_j)=0,$ for each $1 \leq i \neq j \leq n.$ Clearly, $ \{ y^*_1, y^*_2, \ldots, y^*_n\}$ is a basis of $W.$ We show that $ \{ y^*_1, y^*_2, \ldots, y^*_n\}$ is a strong Auerbach basis of $W.$

	Consider  $i_1, i_2, \ldots, i_r \in \{ 1, 2, \ldots, n\}.$    First we claim that the norm attainment set of  $g=\sum_{t=1}^{  r} \delta_{i_t} y^*_{i_t} \in S_{\mathbb{Y}^*},$ for some $\delta_{i_t} \in \mathbb{K} \setminus \{0\},$ is of the form $\{\sum_{t=1}^{r} \lambda_{i_t} y_{i_t}\},$ for some scalars $ \lambda_{i_t} \in \mathbb{K}.$ Since $\mathbb{Y}$ is strictly convex, $M_g$ is unique up to  scalar multiplication. Suppose on the contrary that $M_g= \{  \sum_{t=1}^{r} \lambda_{i_t} y_{i_t}+  \sum_{l=1}^{s}\lambda_{j_l} y_{j_l}\},$  for some $\{j_1, j_2, \ldots, j_s\} \subset \{ 1, 2, \ldots, n\}\setminus \{ i_1, i_2, \ldots, i_r\}$ and $\lambda_{j_l} \in \mathbb{K} \setminus \{0\},$ for any $ 1 \leq l \leq s.$  Observe that $ g(\sum_{t=1}^{r} \lambda_{i_t} y_{i_t} )= g(\sum_{t=1}^{r} \lambda_{i_t} y_{i_t}+  \sum_{l=1}^{s}\lambda_{j_l} y_{j_l} )=1$  and so  $\|\sum_{t=1}^{r} \lambda_{i_t} y_{i_t}\| \geq 1.$ If $\|\sum_{t=1}^{r} \lambda_{i_t} y_{i_t}\|=1,$ it contradicts that $M_g$ is unique up to  scalar multiplication (as $\lambda_{j_l} \neq 0$). If $\|\sum_{t=1}^{r} \lambda_{i_t} y_{i_t}\|> 1,$ then $\| \sum_{t=1}^{r} \lambda_{i_t} y_{i_t}+  \sum_{l=1}^{s}\lambda_{j_l} y_{j_l}\| =1 < \|\sum_{t=1}^{r} \lambda_{i_t} y_{i_t}\|,$ which implies that $ \sum_{t=1}^{r} \lambda_{i_t} y_{i_t} \not\perp_B \sum_{l=1}^{s}\lambda_{j_l} y_{j_l}.$ Clearly, this contradicts the fact that $\mathcal{B}$ is a strong Auerbach basis. This proves our claim.
		Now,  $y^*_k( \sum_{t=1}^{r} \lambda_{i_t} y_{i_t})=0,$ for any $k \in \{1,2, \ldots, n\} \setminus \{ i_1, i_2, \ldots, i_r\},$ and so using  \cite[Th. 3.2]{S2021} it is easy to conclude that 
		\[
		span\{ y^*_{i_1}, y^*_{i_2}, \ldots, y^*_{i_r}\} \perp_B span\{ \{y^*_1, y^*_2, \ldots, y^*_n\} \setminus \{ y^*_{i_1}, \ldots, y^*_{i_r}\}\}.
		\]
	Thus, $\{y^*_1, y^*_2, \ldots, y^*_n\}$ is a strong Auerbach basis of $W.$
	
We next claim that the support functional for an element $ y = \sum_{t=1}^{r} c_{i_t} y_{i_t} \in S_{T(\mathbb{X}_0)},$  $c_{i_t} \in \mathbb{K} \setminus \{0\},$ is of the form $ \sum_{t=1}^{r} d_{i_t} y^*_{i_t},$ for some scalars $d_{i_t} \in \mathbb{K} \setminus \{0\}.$   Suppose on the contrary that  $J(y) = \{ \sum_{t=1}^{r} d_{i_t} y^*_{i_t} + \sum_{l=1}^{s}d_{j_l} y^*_{j_l}\},$ for some $d_{j_l} \in \mathbb{K} \setminus \{0\},$ where $\{j_1, j_2, \ldots, j_s\} \subset \{ 1,2, \ldots, n\} \setminus \{ i_1, i_2, \ldots, i_r\}.$ Then $ \sum_{t=1}^{r} d_{i_t} y^*_{i_t}( y) = (\sum_{t=1}^{r} d_{i_t} y^*_{i_t} + \sum_{l=1}^{s}d_{j_l} y^*_{j_l}) (y)=1.$ Clearly, $\| \sum_{t=1}^{r} d_{i_t} y^*_{i_t}\|\geq 1.$ If $\|  \sum_{t=1}^{r} d_{i_t} y^*_{i_t}\|=1,$ then it contradicts that $J(y)$ is a singleton, as $\mathbb{Y}$ is smooth. If $ \|\sum_{t=1}^{r} d_{i_t} y^*_{i_t}\| > 1,$ then $ \| \sum_{t=1}^{r} d_{i_t} y^*_{i_t} + \sum_{l=1}^{s}d_{j_l} y^*_{j_l} \|=1<\|\sum_{t=1}^{r} d_{i_t} y^*_{i_t}\|, $ which implies $ \sum_{t=1}^{r} d_{i_t} y^*_{i_t} \not\perp_B  \sum_{l=1}^{s}d_{j_l} y^*_{j_l},$ contradicting that $\{y^*_1, y^*_2, \ldots, y^*_n\}$ is a strong Auerbach basis. Thus 
$J(y) = \{ \sum_{t=1}^{r} d_{i_t} y^*_{i_t} \}.$ 
Next, if possible let  $d_{i_r}=0.$ Then $\sum_{t=1}^{r-1} d_{i_t} y^*_{i_t}$ is a support functional of $\sum_{t=1}^{r} c_{i_t} y_{i_t}.$ Since $\mathbb{Y}$ is  strictly convex, $M_{\sum_{t=1}^{r-1} d_{i_t} y^*_{i_t}}$ is unit scalar multiple of $ \sum_{t=1}^{r} c_{i_t} y_{i_t} .$  That contradicts the fact that the norm attainment element of $ \sum_{t=1}^{r-1} d_{i_t} y^*_{i_t}$ is of the form $\sum_{t=1}^{r-1} \sigma_{i_t} y_{i_t}, $ for some $\sigma_{i_t} \in \mathbb{K}.$ Thus the claim is established.
	
	Let $Z = span \bigg\{ (( \alpha_i\beta_j ))_{1 \leq i \leq n, 1 \leq j \leq n} \in \mathbb{K}^{n^2}: \| \sum_{i=1}^n \alpha_i x_i\|=1, \sum_{j=1}^{n}\beta_j y^*_j \in Ext(J(\sum_{i=1}^{n} \alpha_i y_i))
	\bigg\}.$  Then $i_{S_{\mathbb{X}_0}}(T) = dim\, Z.$ Choose $ \alpha_r=1, \alpha_t =0, $ for any $t \in \{1, 2, \ldots, n\} \setminus \{r\}.$ If $\sum_{j=1}^n \beta_j y^*_j$ is a support functional of $ \sum_{i=1}^n \alpha_i T(x_i)= \sum_{i=1}^n \alpha_i y_i= y_r,$ then it is easy to observe that $\beta_r=1, \beta_t=0,$ for any $t \in \{1, 2, \ldots, n\} \setminus \{r\}.$ Therefore, $(0, 0, \ldots, 0,  \alpha_r \beta_r, 0, \ldots, 0)=   \widetilde{e_{rr}} \in Z,$ for any $1 \leq r \leq n.$
	 	Let $ r, s \in \{1, 2, \ldots, n\}$ such that $r \neq s.$ Choose $\alpha_{r} \neq 0, \alpha_{s} \neq 0$ and $\alpha_j=0,$ for any $j \in \{1, 2, \ldots, n\} \setminus \{r, s\}$ such that $\|\sum_{i=1}^n \alpha_i x_i\|=1.$  Then  $J(\sum_{i=1}^n \alpha_i Tx_i)= J(\sum_{i=1}^n \alpha_i y_i)= J( \alpha_{r} y_{r} + \alpha_{s} y_{s}) = \{ \beta_{r} y^*_{r}+ \beta_{s}y^*_{s}\},$ where $\beta_{r}, \beta_{s} \in \mathbb{K} \setminus \{0\}.$ Therefore, $\alpha_{r} \beta_{r} \widetilde{e_{rr}}+ \alpha_{r} \beta_{s} \widetilde{e_{rs}} +  \alpha_{s} \beta_{r} \widetilde{e_{sr}} + \alpha_{s} \beta_{s} \widetilde{e_{ss}} \in Z.$ Then clearly, $    \alpha_{r} \beta_{s} \widetilde{e_{rs}} + \alpha_{s} \beta_{r} \widetilde{e_{sr}} \in Z.$ Therefore, for any $1 \leq i \neq j \leq n,$ $   \alpha_{i} \beta_{j} \widetilde{e_{ij}} +  \alpha_{j} \beta_{i} \widetilde{e_{ji}} \in Z,$ which implies that $dim~Z \geq n+ \frac{n^2-n}{2}= \binom{n+1}{2}.$ Thus $i_{S_{\mathbb{X}_0}}(T) \geq \binom{n+1}{2}$ and so using Theorem \ref{order} we get (i).
	 	
	 	For the proof of (ii), consider  $\mathbb{K}= \mathbb{C}.$ Now if $\alpha_{r} \neq 0, \alpha_{s} \neq 0$ and $\alpha_{t}=0, \forall t \in \{ 1, 2, \ldots, n\} \setminus \{ r, s\},$ then $\alpha_{r} \beta_{s} \widetilde{e_{rs}} + \alpha_{s} \beta_{r} \widetilde{e_{sr}} \in Z.$ Suppose that $ \alpha_{r} \beta_{s}= k \alpha_{s} \beta_{r},$ where $k \in \mathbb{C} \setminus \{0\}.$ Take $\alpha'_{r}= \frac{ \sqrt{-k}}{\|  \sqrt{-k} x_{r}+ x_{s}\|}, \alpha'_{s}= \frac{1}{\|  \sqrt{-k} x_{r}+ x_{s}\|}.$ Clearly $\alpha'_{r} x_{r} + \alpha'_{s} x_{s} \in S_{\mathbb{X}_0}.$ Let $J(\alpha'_{r} y_{r} + \alpha'_{s} y_{s} )= \{\beta'_{r} y^*_{r} + \beta'_{s} y^*_{s}\}.$ This implies $\alpha'_{r} \beta'_{r} + \alpha'_{s} \beta'_{s}=1.$ It is easy to  to verify that 
	 	$\{\alpha_{r} \beta_{s} \widetilde{e_{rs}} + \alpha_{s} \beta_{r} \widetilde{e_{sr}},~ \alpha'_{r} \beta'_{s} \widetilde{e_{rs}} + \alpha'_{s} \beta'_{r} \widetilde{e_{sr}} \}$ is a linearly independent set in $Z,$ and so  $ \widetilde{e_{rs}}, \widetilde{e_{sr}} \in Z.$ Therefore for any $i,j  \in \{1, 2, \ldots, n\},$ $ \widetilde{e_{ij}} \in Z.$ So, $ dim~Z= n^2.$ Thus $i_{S_{\mathbb{X}_0}}(T)  = n^2$ and so using Theorem \ref{order} we get (ii).

	\end{proof}

 Using the  above theorem, the following corollary can be obtained by means of a straightforward calculation.

\begin{cor}\cite[Th. 2.1]{MDP20}
		Let $\mathbb{H}_1, \mathbb{H}_2$ be Hilbert spaces over the field $\mathbb{K}$ and   let  $T \in S_{\mathbb{L}(\mathbb{H}_1, \mathbb{H}_2)}.$   Let $M_{T}$ be the unit sphere of an $n$-dimensional subspace of $\mathbb{H}_1$ and let $\|T\|_{\mathbb{H}_1^\perp} < 1.$  Then
		\begin{itemize}
				\item[(i)] $T$ is  $\binom{n+1}{2}$-smooth, if $\mathbb{K} = \mathbb{R}.$ 
			\item[(ii)]  $T$ is $n^2$-smooth, if $\mathbb{K} = \mathbb{C}.$ 
		\end{itemize}
	\end{cor}

As an application of Theorem \ref{smooth}, we find the $k$-smoothness of some special operators on $\ell_p^n$ spaces. Observe that   the standard ordered basis $\mathcal{B} = \{e_1, e_2, \ldots, e_n\} $ of $\ell_p^n$ is a strong Auerbach basis of $\ell_p^n,$  where $e_i = (0, \ldots,0,1,0 \ldots,0)$ with $1$ in $i$-th position and $0$ elsewhere.

\begin{theorem}\label{l_p}
 Let $T \in \mathbb{L}(\ell_p^n),$ where $1 <p < \infty, \, p \neq 2$ and let $\|T\|=1.$  Suppose that $M_T$ is the unit sphere of an $m$-dimensional subspace $\mathbb{X}$ of $\ell_p^n,$ where $\mathbb{X}= span~ \mathcal{B}_1$ and $\mathcal{B}_1 \subset \mathcal{B}.$    If $T|_{\mathbb{X}}=I|_{\mathbb{X}},$ then $T$ is $m^2$-smooth.
\end{theorem}

\begin{proof} 
	Without loss generality we assume that $\mathcal{B}_1= \{e_1, e_2, \ldots, e_m\}.$ Let $\mathbb{Y} =  span\,\{e_{m+1}, e_{m+2}, \ldots, e_n\}.$   Then $\ell_p^n = \mathbb{X} \oplus \mathbb{Y}$ and clearly, $\mathbb{X} \perp_B \mathbb{Y}.$ 
	 So, from \cite[Lemma 1.3]{MN}, $\mathbb{X}$ is a coproximinal subspace of $\ell_p^n.$
Therefore, $span \{ T(M_T)\}= T(\mathbb{X})= \mathbb{X} = span \{  e_{1}, e_{2}, \ldots, $  $ e_{m}\}$ is a coproximinal subspace of $\ell_p^n$ and it is immediate that $ \{ e_{1}, e_{2}, \ldots,e_{m}\}$ is a strong Auerbach basis of $\mathbb{X}.$ Then from Theorem \ref{smooth},  $T$ is $m^2$-smooth, whenever $\mathbb{K}= \mathbb{C}.$ We just need to show that $T$ is $m^2$-smooth whenever $\mathbb{K}= \mathbb{R}.$  Suppose that $\{e^*_j \} = J( e_{j}),$ for  each $ j, \, 1 \leq j \leq m.$ 
So, we have to show that $dim \,Z=m^2,$ where $Z= span \big\{ (( \alpha_i \beta_j))_{1 \leq i,j \leq m} \in \mathbb{R}^{m^2}: \sum_{i=1}^{m} \alpha_i e_{i} \in M_T,  \{ \sum_{j=1}^{m} \beta_j e^*_j\} = J( \sum_{i=1}^{m} \alpha_i e_{i}) \big\}.$
 It is easy to see  that $ \beta_j = \frac{\alpha_j |\alpha_j|^{p-2}}{(|\alpha_1|^p+ | \alpha_2|^p+ \ldots+ | \alpha_m|^p)^{1- \frac{1}{p}}},$ for any $ 1 \leq j \leq m.$
Following the  arguments  as given in the proof of Theorem \ref{smooth}, we observe that $\widetilde{e_{tt}} \in Z,$ for any $1 \leq t \leq m$  and $ \alpha_r \beta_s  \widetilde{e_{rs}} + \alpha_s \beta_r  \widetilde{e_{sr}} \in Z,$ for any $ 1 \leq r \neq s \leq m,$ which implies,
  $ \frac{\alpha_r \alpha_s |\alpha_r|^{p-2}}{c}  \widetilde{e_{rs}}+ \frac{\alpha_r \alpha_s |\alpha_s|^{p-2}}{c} \widetilde{e_{sr}} \in Z,$ where $c= (|\alpha_1|^p+ | \alpha_2|^p+ \ldots+ | \alpha_m|^p)^{1- \frac{1}{p}}.$ Take $(\alpha'_r, \alpha'_s) \in \mathbb{R}^2 \setminus \{(\alpha_r, \alpha_s),(0,0)\}$ such that $\frac{|\alpha_r|}{|\alpha_s|} \neq \frac{|\alpha_r'|}{|\alpha_s'|}.$ Then it is easy to observe that $\{ \frac{\alpha_r \alpha_s |\alpha_r|^{p-2}}{c}  \widetilde{e_{rs}}+ \frac{\alpha_r \alpha_s |\alpha_s|^{p-2}}{c} \widetilde{e_{sr}},  \frac{\alpha'_r \alpha'_s |\alpha'_r|^{p-2}}{c}  \widetilde{e_{rs}}+ \frac{\alpha'_r \alpha'_s |\alpha'_s|^{p-2}}{c} \widetilde{e_{sr}} \}$ is a  linearly independent set in $Z,$ which implies that  $\widetilde{e_{rs}}, \widetilde{e_{sr}} \in Z,$  for any $1 \leq r, s \leq m.$
Therefore,  $dim ~Z=m^2.$ This completes the theorem. 
\end{proof}

 From Theorem \ref{l_p} and  \cite[Th. 2.1]{MDP20}, the following result is immediate, which characterizes the real Hilbert spaces among the $\ell_p^n(\mathbb{R})$ spaces.
 
\begin{theorem}
 Let $1 <p < \infty.  $ Let $T \in S_{\mathbb{L}(\ell_p^n(\mathbb{R}))}$  be such that $M_T$ is  the unit sphere of an $m$-dimensional subspace of $\ell_p^n(\mathbb{R}).$ Then $T$ is $\binom{m+1}{2}$-smooth if and only if $p=2.$\\
\end{theorem}

The question that arises naturally in this connection is given in the form of the following conjecture.

\begin{conjecture}
	Let $\mathbb{X}$ be a real Banach space. Then the following statements are  equivalent:
	\begin{itemize}
		\item[(i)] $\mathbb{X}$ is a Hilbert space
		\item[(ii)] $T \in S_{\mathbb{L}(\mathbb{X})}$ is $\binom{n+1}{2}$-smooth if and only if $dim~span(M_T)=n,$ where $n \in \mathbb{N}.$ \\  
	\end{itemize}
\end{conjecture}

Theorem \ref{order} can be applied to determine the order of smoothness of certain linear operators. This is illustrated in the following example.

\begin{example}
Consider 	$T \in \mathbb{L}( \ell_p^2(\mathbb{R}), \ell_q^2(\mathbb{R}))$ defined as  $T(x, y)= \frac{1}{2^q} ( x+y, x-y).$ Then $M_T= \{ \pm (1,0), \pm (0, 1), \pm (\frac{1}{2^{\frac{1}{p}}}, \frac{1}{2^{\frac{1}{p}}}), \pm (\frac{1}{2^{\frac{1}{p}}}, -\frac{1}{2^{\frac{1}{p}}})\}.$ Clearly $T(1,0)= (\frac{1}{2^q}, \frac{1}{2^q} ), $ $ T(0,1)=(\frac{1}{2^q}, - \frac{1}{2^q}), T(\frac{1}{2^{\frac{1}{p}}}, \frac{1}{2^{\frac{1}{p}}})=(1,0), T(\frac{1}{2^{\frac{1}{p}}}, -\frac{1}{2^{\frac{1}{p}}})=(0,1).$
Suppose that $J(\frac{1}{2^\frac{1}{q}}, \frac{1}{2^\frac{1}{q}})= \{f_1\}, J(\frac{1}{2^\frac{1}{q}}, - \frac{1}{2^\frac{1}{q}}) = \{ f_2\}, J(1, 0)= \{ f_3\}, J(0, 1)= \{ f_4\}.$ Let $\psi$ be the canonical isometric isomorphism from $(\ell_q^2)^*$ to $\ell_p^2.$ Then it is easy to verify that $ \psi ( f_1)= ( \frac{1}{2^{1- \frac{1}{q}}}, \frac{1}{2^{1- \frac{1}{q}}}), \psi(f_2)= (\frac{1}{2^{1- \frac{1}{q}}}, -\frac{1}{2^{1- \frac{1}{q}}}), \psi (f_3)= (1, 0), \psi (f_4)= (0, 1).$ Take $ \{ (1,0), (0,1)\}$ as a maximal linearly independent set in $M_T.$ Then 
$$i_{M_T}(T) = 
	\dim ~span \{ (1, 0, 0, 0), (0, 0, 0, 1), (\frac{1}{2}, \frac{1}{2}, \frac{1}{2}, \frac{1}{2}), (\frac{1}{2}, - \frac{1}{2}, - \frac{1}{2}, \frac{1}{2})\}= 3.$$ Therefore, from Theorem \ref{order}, $T$ is $3$-smooth.\\
\end{example}	
	
We end this section with the following remark.
 	
	\begin{remark}
			In \cite[p.~ 2]{H}, the author proposed the following problem:
			
			\emph{For Banach spaces $\mathbb{X}$ and $\mathbb{Y},$ let $ T \in  \mathbb{K}(\mathbb{X}, \mathbb{Y})$	with $\|T\|=1$. Is it true that $T$ is a multismooth point of finite
				order $k$ in $\mathbb{K}(\mathbb{X}, \mathbb{Y}) $ if and only if $T^*$ attains its norm at only
				finitely many independent vectors, say at $y_1^*, y_2^*, \ldots, y_r^* \in Ext(B_{\mathbb{Y}^*})$ such that each $T^* y_i^*$ is a multismooth point of finite
				order, say $m_i$, in $\mathbb{X}^*$, where $k=m_1+ m_2 +\ldots+m_r?$}
			
			  In  \cite{W2}, the author showed that    answer to the above query is negative. It should be noted that in the finite-dimensional case, the order of smoothness of $T$ is same as the order of smoothness of  $T^*$ (\cite[Prop. 3.8]{MDP20}).
			  Theorem \ref{l_p} also shows that  the above formulation is incorrect. Indeed,   the cardinality of a maximal linearly independent set in  $M_T$ is $m$  but $T$ is $m^2$-smooth, irrespective of the the underlying field being real or complex. Moreover, in \cite[p.~2]{H}, the author essentially tried to establish a connection between the order of smoothness of an operator $T$ and its norm attainment set $M_T.$ Such a connection has been presented in this paper, by means of the definition of `index of smoothness' (Defn. \ref{definition:index}) in Theorem \ref{order}.
	\end{remark}

\section*{Section-II}

 This section is devoted exclusively to linear operators between finite-dimensional polyhedral Banach spaces. It should be noted that we only consider polyhedral Banach spaces over the real field. We begin this section with following well known definition for the sake of completeness.

	\begin{definition}
		Let $\mathbb{X}$ be an $n$-dimensional Banach space. A convex set $F \subset S_{\mathbb{X}}$ is said to be a face of $B_{\mathbb{X}}$ if for some $x_1, x_2 \in S_{\mathbb{X}},$ $(1-t) x_1+ tx_2 \in F $ implies that $x_1, x_2 \in F,$ where $0 < t< 1.$ The dimension $dim(F)$ of the face $F$ is defined as the dimension of the subspace generated by the differences $v-w$ of vectors $v, w \in F.$ If $dim(F)=i,$ then $F$ is called an \emph{$i$-face} of $B_{\mathbb{X}}.$ $(n-1)$- faces of $B_{\mathbb{X}}$ are called \emph{facets} of $B_{\mathbb{X}}$ and $1$-faces of $B_{\mathbb{X}}$ are called \emph{edges} of $B_{\mathbb{X}}.$ 
	\end{definition}

 Next we characterize $k$-smoothness of an element in a finite-dimensional polyhedral Banach space which extends \cite[Th. 2.1]{DMP}. 

\begin{theorem}\label{interior}
	Let $\mathbb{X}$ be an $n$-dimensional polyhedral Banach space and let $x \in S_{\mathbb{X}}.$ Then  $x$ is  $k$-smooth  if and only if $x$ is an interior point of an $(n-k)$-face of $B_{\mathbb{X}}.$ 
\end{theorem}

\begin{proof}
	We first observe that for each $x \in S_{\mathbb{X}},$ $J(x)$ is convex closed subset of $B_{\mathbb{X}^*}.$  It is easy to observe that $J(x)$ is a face of $B_{\mathbb{X}^*}$, for if 
  $(1-t)x_1^* + t x_2^* \in J(x),$ where $x_1^*, x_2^* \in S_{\mathbb{X}^*}$ then $((1-t)x_1^* + t x_2^*)x=1 \implies x_1^*(x)= x_2^*(x)=1 \implies x_1^*, x_2^* \in J(x).$ So, $Ext(J(x)) \subset Ext(B_{\mathbb{X}^*}).$ 
  	We now prove the necessary part of the theorem.
  	Since  $x$ is  $k$-smooth, $dim~span(Ext(J(x)))=k.$ Let   $f_1, f_2, \ldots, f_k$  be $k$ linearly  independent elements  from  $Ext(J(x)) \subset Ext(B_{\mathbb{X}^*}).$   Let $F_1, F_2, \ldots, F_k$ be the  facets of $B_{\mathbb{X}}$ corresponding to the extreme  functionals  $f_1,f_2, \ldots, f_k$ respectively, see \cite[Lemma  2.1]{SKBS}.  Then $ \cap_{i=1}^k F_i \neq \emptyset,$ as $x \in F_i,$ for $1 \leq i \leq k.$ Clearly, $F=\cap_{i=1}^k F_i$ is a $(n-k)$-face of $B_{\mathbb{X}}.$ We claim that $x \in int(F).$ If not, then  $x \in G,$ where $G$ is a $(n-k-1)$-face of $B_{\mathbb{X}}.$  So there exists a facet $F_{k+1} \not\in \{ \pm F_1, \pm F_2, \ldots, \pm  F_{k}\}$ of $B_{\mathbb{X}}$ such that $G = \cap_{i=1}^{k+1} F_i.$   Let $f_{k+1} $ be the extreme functional corresponding to $F_{k+1}.$  Using \cite[Lemma ~2.1]{SKBS} it is clear that  $f_1, f_2, \ldots, f_{k+1} \in Ext(B_{\mathbb{X}^*})$  are linearly independent. Moreover,  $f_i(x)=1,$ for any $1 \leq i \leq k+1,$ which  shows the existence of $(k+1)$-number of linearly independent support functionals at $x$. This contradicts  the fact that $x$ is a $k$-smooth point. This completes the necessary part of the theorem. The sufficient part follows from similar arguments as given in the proof of  the necessary part.
\end{proof}

We next discuss a step-by-step method to find the order of smoothness of an operator between finite-dimensional polyhedral  spaces.\\

\noindent \textbf{Problem} 	Let $\mathbb{X}, \mathbb{Y}$ be  finite-dimensional polyhedral Banach spaces and let $T \in S_{\mathbb{L}(\mathbb{X}, \mathbb{Y})}.$ Determine the order of smoothness of $T.$ \\

Using the Theorem \ref{order}  we will solve the above problem in the following three steps: \\

\textbf{Step~1:} In the first step we determine $M_T \cap Ext(B_{\mathbb{X}}).$ Let $Ext(B_{\mathbb{X}})= \{ \pm x_1, \pm x_2, $ $ \ldots, \pm x_r\}$ and we determine $Tx_i,$ for any $1 \leq i \leq r$ to check that for which $x_i$'s $\|Tx_i\|= \|T\|.$ Assume that $M_T \cap Ext(B_{\mathbb{X}})= \{ \pm x_1, \pm x_2, \ldots, \pm x_t\}$  and   $\{ x_1, x_2, \ldots, x_k\}$ is a maximal linearly independent set of $M_T \cap Ext(B_{\mathbb{X}}).$ \\

\textbf{Step~2:} Let $Tx_i= y_i,$ for any $1 \leq i \leq t$ and let $y_i$ be $m_i$-smooth, for any $1 \leq i \leq t.$   Suppose that $ \{ y^*_{i1}, y^*_{i2}, \ldots, y^*_{im_i} \} $ is a linearly independent set in $Ext(J(y_i)),$ for any $1 \leq i \leq t.$ It should be noted that as $J(y_i)$ is a face of $B_{\mathbb{Y}^*},$ $\{ y^*_{i1}, y^*_{i2}, \ldots, y^*_{im_i} \} \subset Ext(B_{\mathbb{Y}^*}).$ So, to find a linearly independent set of $Ext(J(y_i)),$ we just need to determine the linearly independent elements $y^*_{ij}$ of $Ext(B_{\mathbb{Y}^*})$ such that  $y^*_{ij}(y_i)= 1.$
 Let $\{ y^*_1, y^*_2, \ldots, y^*_s\}$ be a maximal linearly independent set of $\{ y^* \in S_{\mathbb{Y}^*}: y^* \in J(y_i), 1 \leq i \leq t\}.$ \\
 
 \textbf{Step~3:} Suppose that $x_j= \sum_{i=1}^{k} \alpha_{ij} x_i$   and $y^*_{ij} = \sum_{l=1}^{s} \beta_{l}^{ij} y^*_{l},$  for each $i, \, 1 \leq i \leq k$ and $j, \, 1 \leq j \leq t. $ 
  Then $  i_{M_T}(T) = dim~span Z,$ where $Z= \{ (( \alpha_{qi} \beta^{ql}_j))_{ 1 \leq j \leq s, 1 \leq i \leq k} \in \mathbb{K}^{ks}: 1  \leq l \leq m_i, 1 \leq q \leq t \}.$
      Since $Z \in \mathbb{K}^{ks}$ is finite, it is straightforward to calculate $dim~span~Z,$ and from Theorem \ref{order}, we determine the index of smoothness of $T.$ Moreover, note that  $\{x_1, x_2, \ldots, x_k\}$ is linearly independent and so using Proposition \ref{prop1},  we conclude that $\{ (( \alpha_{qi} \beta^{ql}_j))_{ 1 \leq j \leq s, 1 \leq i \leq k} \in \mathbb{K}^{ks}: 1  \leq l \leq m_i, 1 \leq q \leq k \}$ is a linearly independent set of $Z.$ Thus $i_{M_T}(T) \geq \sum_{i=1}^k m_i.$ \\

Let us now discuss the algorithm given above through an explicit example, to further illustrate its advantage from a computational point of view.

\begin{example}
	Let $\mathbb{X}$ be a $3$-dimensional polyhedral Banach space such that $Ext(B_{\mathbb{X}})= \{ \pm (1, 0, 0), \pm (\frac{1}{\sqrt{2}}, \frac{1}{\sqrt{2}}, 0), \pm (0, 1, 0), \pm (- \frac{1}{\sqrt{2}}, \frac{1}{\sqrt{2}}, 0), \pm(0, 0,1)\}$ and let $\mathbb{Y}= \ell_{\infty}^3.$ Consider  $T \in \mathbb{L}(\mathbb{X}, \mathbb{Y})$   such that  
	\[T(v_1, v_2, v_3)= (v_1+ (\sqrt{2}-1) v_2+ v_3, (\sqrt{2}-1)v_1+ v_2 -v_3, v_1+v_3).\]
	Then  $$M_T \cap Ext(B_{\mathbb{X}})= \{ \pm (1,0,0), \pm (0,1,0), \pm (\frac{1}{\sqrt{2}}, \frac{1}{\sqrt{2}}, 0), \pm (0,0,1)\}.$$ 
	Clearly, $\|T\|=1.$ Let $Z$ be the set defined in the Step $3$ of the above problem. Let $\psi$ be the isometric isomorphism between $(\ell_{\infty}^3)^*$ and $\ell_{1}^3.$ Let $e^*_i \in S_{\mathbb{Y}^*}$ such that $\psi(e^*_i)=e_i,$ where $\{e_1, e_2, e_3\}$ is the standard ordered basis of $\ell_{1}^3.$
	Then clearly, $Ext(B_{\mathbb{Y}^*})= \{ \pm e^*_1, \pm e^*_2, \pm e^*_3\}$  and $\{e^*_1, e^*_2, e^*_3 \}$ is a basis of $\{ f \in J(T(v)): v \in M_T \cap Ext(B_{\mathbb{X}}) \}.$ Then $Z= \{ \widetilde{e_{11}}, \widetilde{e_{13}}, \widetilde{e_{22}}, \widetilde{e_{31}}, -\widetilde{e_{32}}, \widetilde{e_{33}}, \frac{1}{\sqrt{2}} \widetilde{e_{11}}+ \frac{1}{\sqrt{2}} \widetilde{e_{21}}, \frac{1}{\sqrt{2}} \widetilde{e_{11}}+ \frac{1}{\sqrt{2}} \widetilde{e_{22}} \}.$  Clearly, $dim\,Z=7,$ implies that index of smoothness of $T$ with respect to $M_T$ is $7,$ therefore, $T$ is $7$-smooth.
\end{example}

\begin{remark}
	Let $ T\in \mathbb{L}(\ell_1^n,\mathbb{Y}),\|T\|=1.$  In \cite[Cor. 2.3]{MP} it is proved that $T$ is $k-$smooth if and only if $M_T\cap Ext(B_{\ell_1^n})=\{\pm x_1,\pm x_2,$ $\ldots,\pm x_r\}$ for some $1\leq r\leq n,$ $Tx_i$ is $m_i-$smooth for each $1\leq i\leq r$ and $m_1+m_2+\ldots+m_r=k.$  This result now follows easily from the above Problem, since $M_T \cap Ext(B_{\ell_1^n})$ is clearly a linearly independent set 
\end{remark}

The next two results together characterize the order of smoothness of a particular class of operators between polyhedral Banach spaces, which includes the $rank~1$ operators.


\begin{prop}
   	Let $\mathbb{X}, \mathbb{Y}$ be  finite-dimensional polyhedral Banach spaces  such that $\dim ~\mathbb{X}=n, \dim ~\mathbb{Y}=m.$ Let $T \in S_{\mathbb{L}(\mathbb{X}, \mathbb{Y})}$ be such that $M_T= \pm F$ and $|T(M_T)|= 2,$ where $F$ is a face of $B_{\mathbb{X}}.$ If $T$ is $k$-smooth then $ k \in \{ pq: 1 \leq p \leq n, 1 \leq q \leq m\}.$ 
\end{prop}

	\begin{proof}
		Suppose that $T(M_T)= \{ \pm w\}$ and $M_T= \pm F,$ where $F$ is a face of $B_{\mathbb{X}}$ and $w \in S_{\mathbb{Y}}.$ Let us assume that   $\{x_1, x_2, \ldots, x_p\}$ is a maximal linearly independent set of $F \cap Ext(B_{\mathbb{X}}) $ and let $w $ be a $q$-smooth point, where $1 \leq p \leq n, 1 \leq q \leq m.$ Suppose that $y^*_1, y^*_2, \ldots, y^*_q$ are  linearly independent support functionals of $w.$  Then using Lemma \ref{wojcik}, it can be easily verified that  $\{ y^*_i \otimes x_j: 1 \leq i \leq q, 1 \leq j \leq p\}$ is a maximal linearly independent subset of $Ext(J(T)).$ Therefore, 
		$T$ is $pq$-smooth.
	\end{proof}

   	\begin{theorem}\label{face}
   	Let $\mathbb{X}, \mathbb{Y}$ be  finite-dimensional polyhedral Banach spaces such that $\dim ~\mathbb{X}=n, \dim ~\mathbb{Y}=m.$ Let $ k \in \{ pq: 1 \leq p \leq n, 1 \leq q \leq m\}.$ Then  there exists  a $k$-smooth operator $T \in S_{\mathbb{L}(\mathbb{X}, \mathbb{Y})}$  such that $M_T= \pm F$ and $|T(M_T)|= 2,$ where $F$ is a face of $B_{\mathbb{X}}.$
   \end{theorem}
   
   \begin{proof}
   Let $k=p q,$ for some $1 \leq p \leq n, 1 \leq q \leq m.$ We choose a $(p-1)$-face $F$ of $B_{\mathbb{X}}$ and choose $u \in S_{\mathbb{Y}}$ such that $u$ is  $q$-smooth.   Then $dim~span ~J(u)=q.$ Let $y_1^*, y_2^*, \ldots, y_q^* \in Ext \, J(u)$ be such that $ span \, J(u) = span \, \{y_1^*, y_2^*, \ldots, y_q^*  \}. $ Suppose that $F \cap Ext(B_{\mathbb{X}})= \{ x_1, x_2, \ldots, x_r\}.$ It is easy to observe that the  cardinality  of a maximal linearly independent set in  $F \cap Ext(B_{\mathbb{X}})$ is $p.$ Without loss of generality we assume that $ \{x_1, x_2, \ldots, x_{p}\}$ is a linearly independent set in $F \cap Ext(B_{\mathbb{X}}).$ For each  $t,\, p+1 \leq t \leq r,$ let  $x_t= \sum_{j=1}^{p} c_{tj} x_j.$ We claim that $\sum_{j=1}^{p} c_{tj}=1.$ Suppose that $f \in S_{\mathbb{X}^*}$ is such that $f(x)=1,$ for each $x \in F$ and $f(z)<1,$ for any $z \in S_{\mathbb{X}}\setminus F.$ Note that the existence of such a functional is guaranteed from the definition of a face.  Now, $1=f(x_{t})= f( \sum_{j=1}^{p} c_{tj} x_j)= \sum_{j=1}^{p} c_{tj} f(x_j)= \sum_{j=1}^{p} c_{tj},$ for each $t, \, p+1 \leq t \leq r.$
    Let $\mathcal{B}$ be a basis of $\mathbb{X}$ such that $x_j \in \mathcal{B}= \{ x_1, x_2, \ldots, x_{p}, \widehat{x}_{p+1}, \ldots, \widehat{x}_n\},$ for any $j \in \{ 1, 2, \ldots, p\}, $ where $\mathcal{B} \setminus \{ x_1, x_2, \ldots, x_{p}\} \in ker~f.$ Take  $\widetilde{x} \in Ext(B_{\mathbb{X}}) \setminus \{\pm F\}$ and let $\widetilde{x}= \sum_{i=1}^{p} c_i x_i+ \sum_{j=p+1}^{n} c_j \widehat{x}_j,$ where $c_i \in \mathbb{R}.$  Then 
   \[ 1 > f(\widetilde{x})= f\bigg(\sum_{i=1}^{p} c_i x_i+ \sum_{j=p+1}^{n} c_j \widehat{x}_j \bigg)= \sum_{i=1}^{p} c_i f(x_i) + \sum_{j=p+1}^{n} c_j f(\widetilde{x}_j) = \sum_{i=1}^{p} c_i .\]
   So, for any $\widetilde{x} \in Ext(B_{\mathbb{X}}) \setminus \{\pm F\},$ if $\widetilde{x}= \sum_{i=1}^{p} c_i x_i+ \sum_{j=p+1}^{n} c_j \widehat{x}_j$ then $\sum_{i=1}^{p} c_i < 1.$
   Consider  $T \in \mathbb{L}(\mathbb{X}, \mathbb{Y})$ such that $Tx_i=u,$ for any $1 \leq i \leq p$ and $T\widehat{x}_i=0,$ for any $ p+1 \leq i \leq n.$ Observe that for any $1 \leq t \leq r,$ $ \|Tx_t\|= \|\sum_{j=1}^{p} c_{tj} Tx_j\|= \|\sum_{j=1}^{p} c_{tj} u\|= |\sum_{j=1}^{p} c_{tj}|\|u\|=1,$ implies that $ \pm F \subset M_T.$
    Now for any $\widetilde{x} \in Ext(B_{\mathbb{X}}) \setminus \{\pm F\}$ with $\widetilde{x}= \sum_{i=1}^{p} c_i x_i+ \sum_{j=p+1}^{n} c_j \widehat{x}_j, $ it is clear that  $ \| T\widetilde{x}\| = \| T(\sum_{i=1}^{p} c_i x_i+ \sum_{j=p+1}^{n} c_j \widehat{x}_j)\| = \|T(\sum_{i=1}^{p} c_i x_i)\|=\| \sum_{i=1}^{p} c_i u\| \leq | \sum_{i=1}^{p} c_i|< 1.$  This implies $\| T\|=1$ and $M_T= \pm F$. Clearly, $T(M_T) \cap S_{\mathbb{Y}}= \{\pm u\}.$
     Then from Lemma \ref{wojcik}, $Ext(J(T))= \{ y_j^* \otimes x_i: 1 \leq i \leq t, 1 \leq j \leq q\}, $ and by similar arguments given in the proof of Theorem \ref{rank1} we easily obtain that $T$ is $pq$-smooth. This completes the theorem.
   \end{proof}
   
   	\begin{theorem}\label{position:rank1}
   	Let $\mathbb{X}, \mathbb{Y}$ be finite-dimensional polyhedral Banach spaces such that $\dim ~\mathbb{X}=n, \dim ~\mathbb{Y}=m.$ Then there exists a $rank~1$ operator $T \in S_{\mathbb{L}(\mathbb{X}, \mathbb{Y})}$ which is $k$-smooth if and only if $k \in \{ pq: 1 \leq p \leq n, 1 \leq q \leq m\}.$
   	\end{theorem}
   
   \begin{proof}
   	As $rank~T=1,$ we assume that  $ T(M_T)= \{ \pm w\},$ for some $w \in S_{\mathbb{Y}}.$ Now take $S_1= \{ x \in M_T: T(x)=w\}$ and $S_2= \{ x \in M_T: T(x)=-w\}.$ Clearly, $S_2= -S_1.$	It is easy to observe that $S_1$ and $S_2$ are convex. Let $(1-t) x_1 + t x_2 \in S_1, $ where $x_1, x_2 \in S_{\mathbb{X}}.$ Then $ T((1-t) x_1 + t x_2)=w \implies (1-t) Tx_1 + t Tx_2 =w \implies Tx_1 = Tx_2 =w,$ as $rank~T=1.$ This implies that $x_1, x_2 \in S_1$ and consequently, $S_1$ is a face of $B_{\mathbb{X}}.$ So, $M_T= \pm S_1.$ Therefore, from Theorem \ref{face}, we obtain the result immediately.
   \end{proof}

  	The significance of Theorem \ref{position:rank1} is  that  the order of smoothness  of a rank one  operator $T$  does not depend on the geometric structure of the spaces within the realm of polyhedral Banach spaces.  The order of smoothness only depends on the dimension of the  underlying spaces.
  	It is now immediate that  for any integer $k$ less than or  equal to $n,$ where $dim ~\mathbb{X}=n, $ there exists a $rank ~1$ unit norm operator in $\mathbb{L}(\mathbb{X}, \mathbb{Y})$ which is $k$-smooth.   Moreover, Theorem \ref{position:rank1}  provides us with an interesting insight on the positions of the $rank ~1$ operators on the unit ball of $\mathbb{L}(\mathbb{X}, \mathbb{Y}).$
  	 The famous Bertrand's Postulate \cite{MM} (though the result is proved it is still popularly known as a postulate!) ensures that for any $n > 1,$ there always exists a prime number $p$ between $n$ and $2n.$ The following geometric observation is immediate from Theorem \ref{position:rank1}, by using  Theorem \ref{interior} and the Bertrand's Postulate.

   \begin{cor}\label{no-k-smooth}
   	Let $\mathbb{X}, \mathbb{Y}$ be  finite-dimensional polyhedral Banach spaces such that $\dim ~\mathbb{X}=n >1, \dim ~\mathbb{Y}=m>1.$   Then
   	\begin{itemize}
   	\item[(i)]  the $rank~1$ unit norm operators are only in the interior of the $(mn-pq)$-faces of $B_{\mathbb{L}(\mathbb{X}, \mathbb{Y})},$ where $1 \leq p \leq n, 1 \leq q \leq m.$
   	\item[(ii)]  	 for each prime number $p$  with $\max\{m,n\} \leq p \leq mn,$ there does not exist any rank $1$  operator  $T$ in  $S_{\mathbb{L}(\mathbb{X}, \mathbb{Y})}$ such that $T$ is $p$-smooth. 
   	\end{itemize}
    \end{cor}

Since  $(mn-1) \notin \{ pq, 1 \leq p \leq m, 1 \leq q \leq n\},$ whenever $m,n >1,$ using Theorem \ref{interior}, the following geometric observation is immediate.

\begin{cor}
	Let $\mathbb{X}, \mathbb{Y}$ be  finite-dimensional polyhedral Banach spaces such that $ min\{\dim ~\mathbb{X}, \dim ~\mathbb{Y}\}>1.$ Then there exists no $rank~1$ operator in the interior of any edge of the unit ball of $\mathbb{L}(\mathbb{X}, \mathbb{Y}).$
\end{cor}

  In the following observation  we  find the exact  number of faces of $B_{\mathbb{L}(\ell_1^n, \ell_\infty^m)}$ that do not contain a $rank~1$ operator.

   \begin{cor}
   	Let  $S =\{ 1, 2, \ldots, mn\} \setminus   \{ pq: 1 \leq p \leq n, 1 \leq q \leq m\}.$
   	 Then there exists $\sum_{k \in S} \binom{mn}{k}2^{k}$ number of faces of $B_{\mathbb{L}(\ell_1^n, \ell_\infty^m)}$ whose interior does not contain a $rank ~1$  operator.
       \end{cor}
   
   \begin{proof}
   	From \cite[Remark 2.19]{SSP}, we observe that $\mathbb{L}(\ell_1^n, \ell_{\infty}^m)$ is isometrically isomorphic to $\ell_{\infty}^{mn}.$ Let $k \in S.$ Using Theorem \ref{interior} and Theorem \ref{position:rank1},  it suffices to show that the number of $(mn-k)$-face of $B_{\ell_{\infty}^{mn}}$ is $\binom{mn}{k}2^{k}.$ Let $F$ be an $(mn-k )$-face of $B_{\ell_{\infty}^{mn}}.$ Then there exists $Q \subset \{ 1,2, \ldots, mn\}$ such that for any $\widetilde{x}=( x_1, x_2, \ldots, x_{mn}) \in int(F)$ we have $ |x_j|=1, ~j \in  Q$ and $ |x_j|<1,$ for any $j \in \{1,2, \ldots, mn\} \setminus Q.$ Observe that there are exactly  $\binom{mn}{k}$ possible choices of $Q.$ Moreover, each choice of $Q$ corresponds to  $2^{k}$ distinct $(mn-k)$-faces of $B_{\ell_{\infty}^{mn}}.$ Also note that any $(mn-k)$-face of $B_{\ell_{\infty}^{mn}}$ does not correspond to two distinct choice of $Q.$ Therefore, the total number of $(mn-k)$-faces of $B_{\ell_{\infty}^{mn}}$ is $\binom{mn}{k}2^{k}.$ \\
   \end{proof}

We end this article with the following remark which highlights the main findings of the present article. 

\begin{remark}
	Under the condition that $\mathbb{X}$ is reflexive, $\mathbb{K}(\mathbb{X}, \mathbb{Y})$ is an $M$-ideal of $\mathbb{L}(\mathbb{X}, \mathbb{Y})$ and $dist(T, \mathbb{K}(\mathbb{X}, \mathbb{Y})) <1,$ we proved  (see  Theorem \ref{order}) that $T$ is $k$-smooth if and only if $i_{M_T}(T) =k.$ Observe that as given in Definition \ref{definition:index}, determining the value of $i_{M_T}(T)$ is essentially equivalent to determining the dimension of a particular subspace of some $\mathbb{K}^n,$ where $\mathbb{K}=\mathbb{R},$ or, $\mathbb{C}.$ It is clear that the particular subspace under consideration depends on the geometries of both $\mathbb{X}$ and $\mathbb{Y}$. Whenever $\mathbb{X}, \mathbb{Y}$ are finite-dimensional polyhedral Banach spaces, we demonstrate an explicit way to determine the value of $i_{M_T}(T)$ for any $T \in \mathbb{L}(\mathbb{X},\mathbb{Y}).$ This is illustrated step by step in Section-II. Therefore, it follows that the problem of determining the order of smoothness of an operator between finite-dimensional polyhedral Banach spaces can be completely solved in a computationally effective manner by using this newly introduced index of smoothness. In \cite[Th. 2.1]{MDP20}, authors characterized $k$-smoothness of a linear operator defined on a finite-dimensional Banach space. 
	However, determining the value of $i_{M_T}(T)$ helps us to extend \cite[Th. 2.1]{MDP20} to the setting of an arbitrary Banach space.
\end{remark}

\noindent \textbf{Declarations of interest}: None.


\begin{thebibliography}{99} 






\bibitem{BS} Bhatia, R. and  $\check{S}$emrl, P., \textit{Orthogonality of matrices and distance problem},  Linear Algebra Appl., \textbf{287} (1999),  77--85.


\bibitem{B} Birkhoff, G.,  \textit{Orthogonality in linear metric spaces}, Duke Math. J. \textbf{1} (1935), 169--172.




\bibitem{DMP} Dey, S., Mal, A. and  Paul, K., \textit{$k$-smoothness on polyhedral Banach space}, Colloq. Math., \textbf{169} (2022) 25-37.

\bibitem{FF} Franchetti, C. and  Furi, M., \textit{Some characteristic properties of real Hilbert spaces},  Rev. Roumaine Math. Pures. Appl., \textbf{17} (1972), 1045-1048.

\bibitem{MM} Meher, J. and Murty, M. Ram., \textit{Ramanujan's proof of Bertrand's Postulate}, Amer. Math. Monthly, \textbf{120} (2013) 650-653.

\bibitem{H} Hamarsheh, A. S., \textit{ $ k$-smooth points in some Banach spaces}, Int. J. Math. Math. Sci., Vol. $2015$, Article ID $394282$, $4$ pages.

\bibitem{J}James, R. C., \textit{Orthogonality and linear functionals in normed linear spaces}, Trans.  Amer. Math. Soc., \textbf{61} (1947b), 265-292.

\bibitem{KS}  Khalil, R. and  Saleh, A., \textit{ Multi-smooth points of finite order}, Missouri J. Math. Sci., \textbf{17} (2005) 76-87.


\bibitem{KL} Kim, S.K. and Lee, H. J., \textit{The Birkhoff-James orthogonality of operators on infinite dimensional Banach spaces},  Linear Algebra Appl. \textbf{582} (2019), 440-451. 

\bibitem{LR}   Lin, B. and  Rao, T. S. S. R. K.,  \textit{Multismoothness in Banach Spaces}, Int. J. Math. Math. Sci., \textbf{2007} (2007) 12 ~pages.

\bibitem{MPD}  Mal, A.,   Paul. K. and  Dey, S.,   \textit{Characterization of extreme contractions through k-smoothness of operators}, Linear and Multilinear Algebra, (2021) DOI: 10.1080/03081087.2021.1913086

\bibitem{MP}  Mal, A. and   Paul, K., \textit{ Characterization of $k$-smooth operators between Banach spaces}, Linear Algebra Appl., \textbf{586} (2020) 296-307.



\bibitem{MDP20}  Mal, A.,   Dey, S. and  Paul, K., \textit{ Characterization of $k$-smoothness of operators defined between infinite-dimensional spaces}, Linear Multilinear Algebra, (2020) DOI:	10.1080/03081087.2020.18441

\bibitem{MPS} A. Mal, K. Paul and D. Sain, Birkhoff-James orthogonality and geometry of operator spaces,  Infosys Science Foundation Series, Springer Singapore, 2024. ISBN 978-981-99-7110-7, https://doi.org/10.1007/978-981-99-7111-4.

\bibitem{MN} Mazaheri, H. and  Nasri, M., \textit{Complemented subspaces in the normed spaces}, Int. Math. Forum, \textbf{2} (2007), 747-751.

\bibitem{PS} Papini, P. L. and Singer, I., \textit{Best coapproximation in normed linear spaces}, Mh. Math., \textbf{88}  (1979) 27-44.

	\bibitem{P}  Paul, K.,  \textit{Translable radii of an operator in the direction of another operator},  Scientiae Mathematicae, \textbf{2(1)} (1999) pp. 119-122.

\bibitem{PHD} Paul, K.,  Hossein, M. Sk.,  and  Das, K. C., \textit{ Orthogonality on B(H,H) and Minimal-norm  Operator}, Journal of Analysis and Applications,  \textbf{6 } (2008)  169-178.
\bibitem{R} Ryan, R. A., \textit{Introduction to tensor products of Banach spaces} Springer Monographs in Mathematics. Springer-Verlag London, Ltd., London, \textbf{2002} xiv+225 pp. ISBN: 1-85233-437-146B28.


\bibitem{RSDP}  Ray, A.,  Sain, D.,  Dey, S. and  Paul, K., \textit{Some remarks on orthogonality of  bounded linear operators}, J. Convex Anal., \textbf{29} (2022)  No. 1, 165-181. 

\bibitem{RBS} Roy, S., Bagchi, S. and Sain, D., \textit{Birkhoff-James orthogonality in complex Banach spaces and Bhatia-$\breve{S}$emrl Theorem revisited}, Mediterr. J. Math., (2022) 
https://doi.org/10.1007/s00009-022-01979-7.



\bibitem{RRBS} A. Ray, Roy, S., Bagchi, S. and Sain, D., \textit{Extreme contractions on finite-dimensional polygonal Banach spaces. II}, J. Operator Theory, \textbf{84(1)} (2020) 127-137.


\bibitem{S2021}  Sain, D., \textit{On best approximations to compact operators},
Proc. Amer. Math. Soc., \textbf{149} (2021) No. 10, 4273-4286. 

\bibitem{S} Sain, D., \textit{ Birkhoff-James orthogonality of linear operators on finite dimensional Banach
spaces}, J. Math. Anal. Appl., \textbf{447} (2017) 860-866.



\bibitem{S1} Sain, D., \textit{On extreme contractions and the norm attainment set of a bounded linear operator}, Ann. Funct. Anal., \textbf{10(1)} (2019) 135-143.


\bibitem{SP} Sain, D. and Paul, K., \textit{ Operator norm attainment and inner product spaces}, Linear Algebra Appl.,  \textbf{439} (2013) 2448-2452.

\bibitem{SPH}  Sain, D.,   Paul, K. and  Hait, S., \textit{ Operator norm attainment and Birkhoff-James orthogonality}, Linear Algebra Appl., \textbf{476} (2015), 85-97.

\bibitem{SRP} Sain, D., Roy, S. and Paul, K., \textit{A Topological Generalization of Orthogonality in Banach Spaces and some Applications},   J. Convex Anal. \textbf{29} (2022), No. 3, 703--716.


	\bibitem{SRDP}   Sain, D.,  Ray, A.,  Dey, S. and   Paul, K., \textit{Some remarks on orthogonality of bounded linear operators-II}, Acta. Sci. Math. 88 (2022), 807-820.
	
	
	\bibitem{SSP} Sohel, S., Sain, D. and Paul, K., \textit{On subspaces of $\ell_{\infty}$ and extreme contraction in $\mathbb{L}(\mathbb{X}, \ell_{\infty}^n),$}, 
	Monatsh. Math., 202 (2023), 621-636.

\bibitem{SKBS} Sain, D., Paul, K., Bhunia, P. and Bag, S., \textit{On the numerical index of polyhedral Banach space}, Linear Algebra Appl., \textbf{577} (2019), 121-133.

\bibitem{W} W\'{o}jcik, P., \textit{Birkhoff orthogonality in classical $M$-ideals}, J. Aust. Math. Soc. \textbf{103} (2017), 279-288.

\bibitem{W2} W\'{o}jcik, P., \textit{$k$-smoothness: an answer to an open problem}, Math. Scand.,  \textbf{123} (2018), 85-90.

 
\end{thebibliography}
\end{document}